\documentclass[11pt]{elsarticle}
\usepackage{lineno}
\usepackage{amsmath,amssymb,amsthm,color}
\usepackage{enumitem}
\usepackage[title]{appendix}
\usepackage{hyperref}

\newtheorem{theorem}{Theorem}[section]
\newtheorem{lemma}{Lemma}[section]
\newtheorem{proposition}{Proposition}[section]

\newtheorem{remark}{Remark}[section]
\newtheorem{claim}{Claim}
\def\d{{\,\rm d}}
\def\leq{\leqslant}
\def\geq{\geqslant}
\numberwithin{equation}{section}
\journal{arXiv}
\newcommand{\norm}[1]{\left\lVert #1\right\rVert}
\newcommand{\normC}[1]{\left\lVert #1\right\rVert_{\mathcal{C}}}

\begin{document}
	
\begin{frontmatter}
	
	\title{Global exponential turnpike properties for optimal control of the viscous Burgers equation
		\tnoteref{mytitlenote}}
	\tnotetext[mytitlenote]{The last author is supported  by the National Natural Science Foundation of China under grant 12422118.}
	
	\author[mysecondaryaddress]{Emmanuel Tr\'elat}
	\ead{emmanuel.trelat@sorbonne-universite.fr}
	
	\medskip
	
	\author[my3address]{Xingwu Zeng}
\ead{xingwuzeng@whu.edu.cn}

\author[my3address]{Can Zhang}
\ead{canzhang@whu.edu.cn}

\address[mysecondaryaddress]{Sorbonne Universit\'e, CNRS, Universit\'e Paris Cit\'e, Inria, Laboratoire Jacques-Louis Lions, F-75005 Paris, France.}
\address[my3address]{School of Mathematics and Statistics, Wuhan University, Wuhan 430072, China.}

\begin{abstract}
We establish global exponential turnpike properties for quadratic optimal tracking problems governed by the one-dimensional viscous Burgers equation with localized internal control. For every initial datum, finite-horizon optimal solutions approach the unique optimal periodic regime when the periodic tracking target is sufficiently small; the zero-target case yields a global steady turnpike at the origin, with no smallness assumption on the initial datum. To our knowledge, these are the first global exponential turnpike results for the viscous Burgers equation. The proof combines a local exponential turnpike, obtained through strict convexity and periodic Riccati theory, with a parabolic dissipation argument that provides an absorbing time independent of the horizon.
\end{abstract}
	
	\begin{keyword}
viscous Burgers equation \sep exponential turnpike \sep periodic optimal control \sep optimal tracking

\medskip

\MSC[2020]  49N20 \sep 49K20 \sep 93C20 \sep 93D23

\end{keyword}

\end{frontmatter}



\section{Introduction}
\subsection{Motivation and main ideas}
The turnpike phenomenon refers to the fact that, over a sufficiently long time horizon, optimal trajectories remain for most of the time close to a distinguished steady state or reference trajectory, up to possible boundary layers near the initial and terminal times. 
To the best of our knowledge, this phenomenon was first observed in economics (see \cite{ML, SP}). 
Later, it has been extensively analyzed in the context of optimal control problems, in the deterministic case (see \cite{GG1, GSS, GTZ, PZ13, PZ, T1, TZZ, TZ_MCSS, TZZpafa, TZZ18,  TZ, ZS, za}) as well as in the stochastic case (see \cite{SWY, SY_JDE}). 
Most of the available results concern linear problems, especially of LQ type.

The turnpike property for nonlinear control problems is much more delicate.  
In the finite-dimensional setting, nonlinear turnpike properties were investigated in \cite{TZ}. 
For semilinear optimality systems, the existence of solutions satisfying a turnpike property was established in \cite{PZ}. 
A local steady-state turnpike result under smallness assumptions on the initial datum and on the static adjoint state was proved in \cite{TZZ18}. 
The competition between multiple static turnpike candidates was studied in \cite{Rapaport04,Rapaport05}; see also \cite{PD_EMS} for the influence of multiplicity of minimizers on the turnpike behavior in semilinear elliptic optimal control. 
For a class of nonlinear optimal control problems with a steady control-state tracking term, turnpike properties are established as well in \cite{EGPZ}. 
A global exponential turnpike property is proved in \cite{PD} for a semilinear control problem with a sign condition, where ``global'' means that no smallness condition is imposed on the initial datum of the state equation.

In this paper, we study the global turnpike property for the optimal control problem governed by the viscous Burgers equation. 
More precisely, we consider a large-time optimal tracking problem $ (OCP)_T $ and its periodic version $ (Per)_\theta $, where the tracking term $y_d \in C([0,\infty);L^2(0,1))$ is $\theta$-periodic in time and the states satisfy controlled Burgers equations. 

The main novelties and contributions of the present paper are twofold: we provide the first global exponential turnpike result for the viscous Burgers equation, and we further extend the analysis to the periodic tracking setting.
In this sense, we extend the local turnpike analysis for the heat equation with a convective nonlinearity, as discussed in \cite[Section 9.1]{EGPZ} (see also \cite{ZS}), to a global framework.
More specifically, we prove that: i) when $y_d\not\equiv0$ is sufficiently small, $(Per)_\theta$ has a unique $\theta$-periodic minimizer and every minimizer of $(OCP)_T$, starting from an arbitrary initial datum $y_0\in L^2(0,1)$, satisfies a global exponential turnpike estimate around that periodic minimizer (see Theorem~\ref{main1}); ii) when $y_d\equiv0$, every minimizer of $(OCP)_T$, again starting from an arbitrary initial datum, satisfies a global exponential turnpike estimate around the origin (see Theorem~\ref{main0}).

We emphasize that the global turnpike property arises from both the smallness condition on $y_d$ and the parabolic dissipation, each playing a different role. 
On the one hand, the smallness condition on $y_d$ is essential to overcome the non-convexity introduced by the Burgers nonlinearity. 
If this smallness condition is removed, the periodic problem $ (Per)_\theta $ may admit multiple optimal periodic pairs, leading to a ``competition phenomenon'' at the turnpike level.
In this situation, one may expect optimal solutions of $(OCP)_T$ to exhibit a turnpike behavior toward a set of periodic minimizers, with the selected member determined asymptotically by the entry and exit costs, provided that each candidate is locally exponentially stable.
The smallness assumption excludes this possible competition by ensuring strict coercivity of the second variation and hence the uniqueness of the periodic turnpike.
On the other hand, the parabolic dissipation yields a uniform absorbing mechanism, which allows trajectories issued from arbitrary initial data to enter, after a time independent of the horizon $T$, a neighborhood of the unique periodic optimal pair. 
The combination of these two features leads to the global exponential turnpike property proved in this paper. 

\subsection{Formulation of the optimal periodic tracking problem}
The Burgers equation is the simplest one-dimensional nonlinear evolution partial differential equation for convection-diffusion phenomena (see \cite{BJM0, BJM1}); the associated control problems have been widely studied in the literature (see \cite{CoronXiang, FG, FI, Volkwein2001}).
Long-horizon inverse design for the Burgers equation has also been investigated numerically in \cite{APZ}; that work optimizes the initial datum against a terminal target and is therefore distinct from the distributed tracking problem and the analytic turnpike results considered here.

We first introduce some notation. 
Let $ T>0 $ be the time horizon and  $ \omega \subset (0,1) $ be an open and non-empty subset with $\chi_\omega$ being its characteristic function.  
We identify controls in $L^2(\omega)$ with their extensions by zero to $(0,1)$; in particular, products such as $\chi_\omega\lambda$ are understood as elements of $L^2(0,1)$, while their control norm is taken in $L^2(\omega)$.
Throughout the work, $ H_0^1 (0,1) $ is endowed with the inner product $ \langle f,g\rangle_{H_0^1(0,1)} = \int_{0}^{1}\partial_xf(x)\partial_xg(x)\d x,\; \forall f,g\in H_0^1(0,1) $, and $ W(0,T) := \{\varphi\in L^2(0,T;H_0^1 (0,1))\,\mid\,\partial_t\varphi \in L^2(0,T;H^{-1}(0,1))\}$ is endowed with the norm
$$ 
\|\varphi\|_{W(0,T)} := \big(\|\varphi\|_{L^2(0,T;H_0^1 (0,1))}^2+\|\partial_t\varphi\|_{L^2(0,T;H^{-1}(0,1))}^2\big)^{1/2}.
$$
It is easy to check that $ W(0,T) $ is a Hilbert space and embeds continuously into $ C([0,T];L^2(0,1)) $ (see \cite{DL}), and that $ W(0,T) $ embeds compactly into $ L^2(0,T;L^{\infty}(0,1)) $ (see \cite{TR}). 
When no confusion arises, a function $ f(t,x) $ of two variables $ (t,x) $ can also be regarded as a mapping $ t\mapsto f(t) $, where for each $ t $, $ f(t) $ denotes the function $ x\mapsto f(t,x) $, understood in a suitable space. 
We use $C(\cdot), C_1(\cdot),C_2(\cdot),\ldots $ to denote  generic positive constants depending on the quantities listed in the parentheses, which may change from line to line.

For every $ y_0  \in L^2(0,1)$ and $y_d\in C([0,\infty);L^2(0,1))$, we consider the optimal tracking problem
\begin{equation*}\label{12181}
	(OCP)_T:\; \inf_{u\in L^2(0,T;L^2 (\omega))}\;\;\frac{1}{2}\int_0^T 
	\big(\| y(t)-y_d(t) \|_{L^2 (0,1)}^2+ \|u(t)\|_{L^2 (\omega)}^2\big)\d t,
\end{equation*}
where the tracking term $y_d(\cdot) $ is $\theta$-periodic in $ t $ with $\theta>0$, i.e., $y_d(t+\theta)=y_d(t),\; \forall t \geq 0$, and $y(\cdot)\in L^2(0,T;L^2 (0,1)) $ is the solution of the controlled Burgers equation 
\begin{equation}\label{state}
	\left\{
	\begin{aligned}
		&\partial_ty(t,x) -  \partial_{xx}y(t,x) + y(t,x)\partial_x y (t,x) = \chi_\omega u(t,x) , \\
		&y(t,0)=y(t,1)=0,\\
		&y(0,x) = y_0 (x),
	\end{aligned}
	\right.
\end{equation}
for $(t,x)\in (0,T)\times(0,1)$, with $ u(\cdot)\in L^2(0,T;L^2 (\omega))$.
It is well known that the equation~\eqref{state} has a unique weak solution\footnote{We say that $ y\in W(0,T) $ is a weak solution to \eqref{state}, if $ \frac{\d}{\d t}\langle y(t), \varphi\rangle +\langle \partial_xy(t), \partial_x\varphi \rangle+\langle y(t)\partial_xy(t),  \varphi \rangle=\langle \chi_\omega u(t), \varphi\rangle $ a.e. $ t\in[0,T] $ for all $ \varphi\in H_0^1(0,1) $, and $ y(0)=y_0\in L^2(0,1), $ where $\langle\cdot,\cdot\rangle$ is the inner product in $ L^2(0,1) $. } $ y \in W(0,T) $ (see \cite[Theorem 3]{Volkwein2001}).

We define the associated periodic optimal control problem
\begin{equation*}\label{xiao-2}
	(Per)_\theta:\; \inf\;\;  
	\frac{1}{2}\int_0^\theta 
	\big(\| y(t)-y_d(t) \|_{L^2 (0,1)}^2+  \|u(t)\|_{L^2 (\omega)}^2\big)\d t,
\end{equation*}
over all $(y(\cdot),u(\cdot))\in W(0,\theta)\times L^2(0,\theta;L^2 (\omega))$ satisfying
\begin{equation}\label{per_state}
	\left\{
	\begin{aligned}
		&\partial_ty(t,x) -  \partial_{xx}y(t,x) + y(t,x) \partial_xy (t,x) = \chi_\omega u(t,x) ,\\
		&y(t,0)=y(t,1)=0,\\
		&y(0,x) = y(\theta,x),
	\end{aligned}
	\right.
\end{equation}
for $(t,x)\in (0,\theta)\times(0,1)$.
Since $y_d$ is continuous and $\theta$-periodic in time, we abbreviate, here and throughout,
\begin{equation}\label{def:normC}
	\normC{y_d} := \|y_d\|_{C([0, \infty);L^2 (0,1))} = \max_{t\in[0,\theta]}\|y_d(t)\|_{L^2(0,1)}.
\end{equation}

\subsection{Main results}
The main result is the following global exponential turnpike property. 
In the following theorems, we denote by $\left(y_{T,*}(\cdot), u_{T,*}(\cdot)\right)$ any optimal pair of $(OCP)_T$, and by $\left(y_{\theta,*}(\cdot), u_{\theta,*}(\cdot)\right)$ the unique optimal pair of $(Per)_\theta$, extended to the whole real line by $\theta$-periodicity.
\begin{theorem}[Periodic turnpike]\label{main1}
	There exists $ \rho >0$ such that, for any $\theta$-periodic tracking term $y_d\in C([0, \infty);L^2 (0,1))$ satisfying
	\begin{equation}\label{smallness_target_varepsilon}
		\normC{y_d} \leq \rho,
	\end{equation} $(Per)_\theta$ admits a unique optimal pair $\left(y_{\theta,*}(\cdot), u_{\theta,*}(\cdot)\right)$, and
there exists $\mu:=\mu(\theta,\omega,y_d)>0$ such that, for every $ y_0\in L^2(0,1) $, there are constants $K:=K(\theta,\omega,\|y_0\|_{L^2(0,1)},y_d)>0$ and $\tau:=\tau(\theta,\omega,\|y_0\|_{L^2(0,1)},y_d)>0$ such that, for all $T> \tau$ and a.e.~$t \in[\tau,T]$,
	\begin{equation}\label{turnpike1}
		\|y_{T,*}(t) - y_{\theta,*}(t)\|_{L^2(0,1)} + \|u_{T,*}(t) - u_{\theta,*}(t)\|_{L^2(\omega)}\leq K\big(e^{-\mu t }+e^{-\mu(T-t)} \big).
	\end{equation}
\end{theorem}
When $y_d\equiv0$, the constant function $y_d$ is $\theta$-periodic for every $\theta>0$; applying Theorem~\ref{main1} with $\theta=1$, the smallness assumption \eqref{smallness_target_varepsilon} is satisfied and the unique optimal triple $\left(y_{\theta,*}(\cdot),u_{\theta,*}(\cdot),\lambda_{\theta,*}(\cdot)\right) $ of $ (Per)_\theta $ reduces to $\left(0,0,0\right)$. 
Fixing $\theta=1$ and absorbing this fixed choice into the constants, the resulting constants depend only on $\omega$ and $\|y_0\|_{L^2(0,1)}$. Hence, as a direct corollary of Theorem~\ref{main1}, we obtain the following global exponential steady turnpike property.
\begin{theorem}[Steady turnpike]\label{main0}
		If $y_d\equiv0$, then there exists $\mu:=\mu(\omega)>0$ such that, for every $ y_0\in L^2(0,1) $, there are constants $K:=K(\omega,\|y_0\|_{L^2(0,1)})>0$ and $\tau:=\tau(\omega,\|y_0\|_{L^2(0,1)})>0$ such that, for all $T> \tau$ and a.e.~$t \in[\tau,T]$,
		\begin{equation*}\label{steady_turnpike1}
			\|y_{T,*}(t)\|_{L^2(0,1)} + \|u_{T,*}(t) \|_{L^2(\omega)}\leq K\big(e^{-\mu t }+e^{-\mu(T-t)} \big).
		\end{equation*}
\end{theorem}
Several remarks are in order.
\begin{remark}\label{rem:constants}
	Theorems~\ref{main1} and \ref{main0} provide the first global exponential turnpike results for the Burgers control problem. The proof uses an exponential stability rate $\mu>0$ for the Riccati closed-loop evolution operator in \eqref{prop2}. After the admissible target radius has been fixed sufficiently small, $\mu$ is chosen uniformly on that target ball and under time shifts; a sharper rate for a fixed problem may of course depend on the full periodic optimal pair, and hence on the profile of $y_d$. The constants $\tau$ and $K$ depend on $\theta$, $\omega$, $y_d$ and, in addition, on $\left\|y_0\right\|_{L^2(0,1)}$; none of these constants depends on $T$.
\end{remark}
\begin{remark}\label{rem:comparisonLQ}
	Compared with our previous work \cite{TZZ} on linear-quadratic problems, the periodic turnpike estimate established in Theorem~\ref{main1} requires a small tracking term $y_d$ and holds only on the time interval $[\tau,T]$, excluding an initial short segment. Although the proof provides a uniform positive lower bound for the decay rate on a sufficiently small target ball, the sharper rate attached to a fixed periodic closed loop may depend on the profile of $y_d$, and the dependence of $K$ on the initial datum $y_0$ is not as explicit as in \cite{TZZ}. These differences are due to the nonlinearity of the equation.
\end{remark}
\begin{remark}\label{rem:mechanism}
	The global nature of Theorem~\ref{main1} comes from the combination of local coercivity and parabolic dissipation. The smallness of $y_d$ ensures the uniqueness of the periodic optimal pair $(y_{\theta,*},u_{\theta,*})$, the exponential stability of the periodic Riccati closed loop with rate $\mu>0$, and a local turnpike result (Proposition~\ref{local_turnpike}). Dissipation yields a uniform absorbing time $\tau$, independent of $T$, after which the trajectory enters the local critical regime (Lemma~\ref{lemma_absorbing_time}); on the remaining interval the local turnpike estimate applies.
\end{remark}
\begin{remark}\label{rem:steadysimpler}
	In the steady case $y_d\equiv0$ of Theorem~\ref{main0}, the proof simplifies substantially. Indeed, comparing any optimal control with the null control and using the free exponential decay of the uncontrolled Burgers equation yields $\|u_{T,*}\|_{L^2(0,T;L^2(\omega))}^2 \leq 2J_T(u_{T,*})\leq 2J_T(0)\leq \frac{1}{2\pi^2}\|y_0\|_{L^2(0,1)}^2$, a bound which is uniform in $T$; hence the uniform slice conclusion \eqref{uniform_bound_optima1} holds directly in this case, without a smallness condition on $y_0$, and the controllability arguments of Appendix~\ref{app:A} are not needed. We do not develop this shortcut separately, since Theorem~\ref{main0} follows anyway from Theorem~\ref{main1}.
\end{remark}
\begin{remark}\label{rem:openproblem}
	Fix $\theta>0$ and let the target $y_d=y_d(x)\in L^2(0,1)$ be steady and sufficiently small for Proposition~\ref{local_uniqueness} to apply. Since the problem is invariant under time translation, each time translate of its unique periodic optimal pair is again optimal; uniqueness therefore forces this pair to be steady. It follows that it is the unique minimizer of the static optimization problem
	\begin{equation*}
		\inf \big(\| y-y_d \|_{L^2 (0,1)}^2+ \|u\|_{L^2 (\omega)}^2\big)
		\quad\text{subject to}\quad
		-  \partial_{xx}y + y\,\partial_x y = \chi_\omega u \text{ in } (0,1),\ y(0)=y(1)=0.
	\end{equation*}
	Indeed, every static admissible pair is periodic, while the preceding translation argument shows that the periodic minimizer is static. When the smallness assumption on $y_d$ is dropped, the static problem may admit several minimizers (see \cite{PD_EMS} for related nonuniqueness phenomena in semilinear elliptic optimal control), and one may then expect a competition between several turnpike candidates, in the spirit of \cite{Rapaport04,Rapaport05,T1}: an optimal trajectory of $(OCP)_T$ could approach the \emph{set} of static minimizers, with the selected member determined by the entry and exit layers. Establishing a global turnpike property in this multi-attractor regime for the Burgers equation is an interesting open problem.
\end{remark}

The paper is organized as follows. In Section~\ref{sec:pre}, we establish uniform a priori estimates for the optimality systems and for optimal solutions, and we prove the local uniqueness of the minimizers of $(OCP)_T$ and $(Per)_\theta$. In Section~\ref{sec:periodicturnpike}, we state a local turnpike result, prove two cost comparison lemmas and a global absorbing property, and combine them to prove Theorem~\ref{main1}. The proofs of the three key technical results (uniform slice bounds, strict convexity, local turnpike) are given in Appendices~\ref{app:A}, \ref{app:B} and \ref{app:C}, respectively.

\section{Preliminaries}\label{sec:pre}
\subsection{Auxiliary inequalities and notation}\label{sec:notation}
Given $\bar T>0$ and $T>0$, we set $ N_T := \lfloor T/\bar{T}\rfloor $ (integer part) and we consider the partition of $ (0,T) $ into the slices
\begin{equation}\label{def:slices}
	I_i := \big((i-1)\bar{T},\, i\bar{T}\big)\quad\text{for } 1\leq i\leq N_T,
	\qquad
	I_{N_T+1} := \big(N_T\bar{T},\, T\big),
\end{equation}
the last slice being empty when $ T=N_T\bar{T} $. The value of $\bar{T}$ is specified wherever the slices are used ($\bar T =4$ in Lemma~\ref{uniform_bound_optima} and Appendix~\ref{app:A}, $\bar T =1$ in Appendix~\ref{app:B}).

We shall repeatedly use the following elementary inequalities:
\begin{itemize}
	\item[(i)] (Agmon) $ \norm{f}_{L^\infty(0,1)}^2\leq 2 \norm{f}_{L^2(0,1)}\norm{\partial_x f}_{L^2(0,1)} $ for every $f\in H_0^1(0,1)$;
	\item[(ii)] (Poincar\'e) $ \pi^2\norm{f}_{L^2(0,1)}^2\leq \norm{\partial_x f}_{L^2(0,1)}^2 $ for every $f\in H_0^1(0,1)$;
	\item[(iii)] (Young) $ ab\leq \frac{3\varepsilon^{4/3}}{4}\,a^{4/3}+\frac{1}{4\varepsilon^{4}}\,b^{4} $ for all $a,b\geq0$ and $\varepsilon>0$;
	\item[(iv)] for every $\varepsilon>0$, every $u\in L^2(0,1)$ and every $v\in H_0^1(0,1)$,
	\begin{equation}\label{ineq:corecive}
		\left|\int_0^1 u\, v\, \partial_x v  \d x \right|
		\leq \frac{3\sqrt{2}\,\varepsilon^{4/3}}{4}\norm{u}_{L^2(0,1)}\norm{\partial_x v}_{L^2(0,1)}^2
		+\frac{\sqrt{2}}{4\varepsilon^{4}}\norm{u}_{L^2(0,1)}\norm{v}_{L^2(0,1)}^2 .
	\end{equation}
\end{itemize}
The estimate \eqref{ineq:corecive} follows from (i) and (iii): indeed,
$\big|\int_0^1 u v \partial_x v\big| \leq \norm{u}_{L^2}\norm{v}_{L^\infty}\norm{\partial_x v}_{L^2}\leq \sqrt{2}\norm{u}_{L^2}\,\norm{v}_{L^2}^{1/2}\norm{\partial_x v}_{L^2}^{3/2} $,
and it remains to apply (iii) with $a=\norm{\partial_x v}_{L^2}^{3/2}$ and $b=\norm{v}_{L^2}^{1/2}$.
We refer to \eqref{ineq:corecive} as the \emph{coercivity estimate}: combined with (ii), it shows that the quadratic form $v\mapsto \norm{\partial_x v}_{L^2}^2-\int_0^1 uvv_x$ remains coercive as long as $\norm{u}_{L^2(0,1)}$ is small enough.
All constants below are independent of the time horizon $T$.

\subsection{Optimality systems and a priori estimates}\label{subsec:OS}
The problem $ (OCP)_T $ has at least one optimal solution, denoted by $ \left(y_{T,*}(\cdot),u_{T,*}(\cdot)\right) $, which, along with an adjoint state $ \lambda_{T,*}(\cdot) \in W(0,T) $, solves the following optimality system with $\left(y_{T}(\cdot),u_{T}(\cdot),\lambda_{T}(\cdot)\right)=\left(y_{T,*}(\cdot),u_{T,*}(\cdot),\lambda_{T,*}(\cdot)\right)$ (see \cite[Theorem 6]{Volkwein2001})
\begin{equation}\label{maxim}
	\left\{
	\begin{aligned}
		&\partial_t y_T (t,x) - \partial_{xx} y_T (t,x) + y_T (t,x)\partial_x y_T (t,x) = \chi_\omega \lambda_T(t,x) , \\
		&-\partial_t \lambda_T (t,x) - \partial_{xx} \lambda_T (t,x) - y_T (t,x) \partial_x \lambda_T (t,x) = - (y_T (t,x)-y_d (t,x)), \\
		&y_T (t,0)=y_T (t,1)=0,\quad \lambda_T (t,0)=\lambda_T (t,1)=0,\\
		&y_T (0,x) = y_0 (x),\quad \lambda_T (T,x)=0,
	\end{aligned}
	\right.
\end{equation}
for $(t,x)\in (0,T)\times (0,1)$,
and
\begin{equation}\label{contrT}
	u_T (t,x) = \chi_\omega\lambda_T(t,x), \quad \text{a.e. } (t,x)\in (0,T)\times (0,1).
\end{equation}

We first state a basic energy estimate. We give it for an arbitrary control, since it will also be applied, in Appendix~\ref{app:A}, to competitors that are not related to any optimality system.
\begin{lemma}\label{energyT}
	Let $T>0$, $y_0\in L^2(0,1)$, $u\in L^2(0,T;L^2(\omega))$, and let $y$ be the solution of \eqref{state}. Then, for all $ 0\leq t_0< t_1 \leq T $,
	\begin{equation}\label{lemma1.2.}
		\sup_{t\in[t_0,t_1]}\|y(t)\|_{L^2(0,1)} + \|y\|_{L^2(t_0,t_1;H_0^1 (0,1))}\leq 2\left(\|y(t_0)\|_{L^2(0,1)} + \|u\|_{L^2(t_0,t_1;L^2 (\omega))}\right).
	\end{equation}
	In particular, \eqref{lemma1.2.} holds for the state $y_T$ of \eqref{maxim}--\eqref{contrT}, with $u=\chi_\omega\lambda_T$.
\end{lemma}
\begin{proof}
	Multiplying the equation by $y$ and using $\int_0^1 y\,\partial_x y\, y \d x=\frac13\left(y^3(t,1)-y^3(t,0)\right)=0$, we obtain
	\begin{equation*}
		\frac{1}{2}\frac{\d}{\d t}\norm{y}_{L^2(0,1)}^2+\norm{\partial_x y}_{L^2(0,1)}^2=\int_\omega uy
		\leq \frac{1}{2}\norm{u}_{L^2(\omega)}^2+\frac{1}{2}\norm{y}_{L^2(0,1)}^2
		\leq \frac{1}{2}\norm{u}_{L^2(\omega)}^2+\frac{1}{2}\norm{\partial_x y}_{L^2(0,1)}^2,
	\end{equation*}
	by the Poincar\'e inequality (ii). Hence
	\begin{equation}\label{diffineq}
		\frac{\d}{\d t}\norm{y}_{L^2(0,1)}^2+\norm{\partial_x y}_{L^2(0,1)}^2 \leq \norm{u}_{L^2(\omega)}^2 .
	\end{equation}
	Integrating \eqref{diffineq} on $[t_0,t]$ for $t\in[t_0,t_1]$ shows that both $\sup_{[t_0,t_1]}\|y(t)\|^2_{L^2(0,1)}$ and $\|y\|^2_{L^2(t_0,t_1;H_0^1(0,1))}$ are bounded above by $\|y(t_0)\|^2_{L^2(0,1)}+\|u\|^2_{L^2(t_0,t_1;L^2(\omega))}$, whence \eqref{lemma1.2.}.
\end{proof}

The next lemma bounds the state, uniformly in $T$, in terms of the norms of the control on the slices \eqref{def:slices}.
\begin{lemma}\label{lem_y_upper_bound_by_sub_inter}
	Let $\bar T>0$ and set
	\begin{equation}\label{def:K}
		K_{\mathrm{sl}}:=K_{\mathrm{sl}}(\bar{T}):=\left(\frac{2\big(2-e^{-\pi^2\bar{T}}\big)}{1-e^{-\pi^2\bar{T}}}\right)^{1/2}\ (\geq\sqrt{2}).
	\end{equation}
	For every $T>0$, $y_0\in L^2(0,1)$ and $u\in L^2(0,T;L^2(\omega))$, the solution $y$ of \eqref{state} satisfies
	\begin{equation}\label{y_upper_bound_by_sub_inter}
		\|y\|_{C([0,T];L^2(0,1))}\leq K_{\mathrm{sl}}\Big(\|y_0\|_{L^2(0,1)} + \sup_{1\leq i\leq N_T+1}\|u\|_{L^2(I_i;L^2 (\omega))}\Big).
	\end{equation}
\end{lemma}
\begin{proof}
	By \eqref{diffineq} and the Poincar\'e inequality,
	$ \frac{\d}{\d t}\norm{y}_{L^2(0,1)}^2+\pi^2\norm{y}_{L^2(0,1)}^2 \leq \norm{u}_{L^2(\omega)}^2 $,
	whence, for every $t\in(0,T]$,
	\begin{equation}\label{y_L2}
		\|y(t)\|_{L^2(0,1)}^2\leq e^{-\pi^2 t}\|y_0\|_{L^2(0,1)}^2+\int_0^t e^{-\pi^2 (t-s)}\|u(s)\|_{L^2(\omega)}^2 \d s .
	\end{equation}
	Set $S:=\sup_{1\leq i\leq N_T+1}\|u\|_{L^2(I_i;L^2(\omega))}$ and $n:=\lfloor t/\bar T\rfloor$. For $s\in I_i$ with $i\leq n$ one has $e^{-\pi^2(t-s)}\leq e^{-\pi^2(t-i\bar T)}\leq e^{-\pi^2(n-i)\bar T}$, while $e^{-\pi^2(t-s)}\leq1$ on $(n\bar T,t)$. Hence
	\begin{equation}\label{int_u}
		\int_0^t e^{-\pi^2 (t-s)}\|u(s)\|_{L^2(\omega)}^2 \d s
		\leq \Big(\sum_{k= 0}^{\infty}e^{-\pi^2 k\bar T }+1\Big)S^2
		=\frac{2-e^{-\pi^2\bar{T}}}{1-e^{-\pi^2\bar{T}}}\,S^2 .
	\end{equation}
	Combining \eqref{y_L2}--\eqref{int_u} gives 
	$$
	\|y(t)\|^2_{L^2(0,1)}\leq \frac{2-e^{-\pi^2\bar{T}}}{1-e^{-\pi^2\bar{T}}}\left(\|y_0\|^2_{L^2(0,1)}+S^2\right)\leq \frac{K_{\mathrm{sl}}^2}{2}\left(\|y_0\|_{L^2(0,1)}+S\right)^2,
	$$
	whence \eqref{y_upper_bound_by_sub_inter}.
\end{proof}

We now derive uniform bounds for the adjoint state of the optimality system \eqref{maxim}--\eqref{contrT}, under a smallness condition on the initial datum and on the slice norms of the adjoint state.
\begin{lemma}\label{regular_lambda}
	Let $\bar T>0$, let $K_{\mathrm{sl}}=K_{\mathrm{sl}}(\bar T)$ be as in \eqref{def:K}, and let $R_1 \in \big(0, \frac{1}{4\sqrt{2}K_{\mathrm{sl}}}\big)$. Let $(y_T,u_T,\lambda_T)$ be a solution of \eqref{maxim}--\eqref{contrT} such that
	\begin{equation}\label{slice-small}
		\|y_0\|_{L^2(0,1)}\leq R_1\qquad\text{and}\qquad \sup_{1\leq i\leq N_T+1}\|\lambda_T\|_{L^2(I_i;L^2 (\omega))}\leq R_1 .
	\end{equation}
	Then the following statements hold.
	\begin{itemize}
		\item[(i)] For all $ 0\leq t_0\leq t_1 \leq T $,
		\begin{equation}\label{rho_exp_decay}
			\|\lambda_T(t_0)\|_{L^2(0,1)}^2\leq e^{-\frac{t_1-t_0}{2}}\|\lambda_T(t_1)\|_{L^2(0,1)}^2+2\int_{t_0}^{t_1} e^{-\frac{s-t_0}{2}}\,\|y_T(s)-y_d(s)\|_{L^2(0,1)}^2 \d s .
		\end{equation}
		\item[(ii)] Setting $K_1:=K_1(\bar T):=K_{\mathrm{sl}}\left(\frac{2\bar T}{1-e^{-\bar T/2}}\right)^{1/2}$,
		\begin{equation}\label{eq_3.20}
			\|\lambda_T\|_{C([0,T];L^2(0,1))}\leq K_1\Big(\|y_0\|_{L^2(0,1)} + \sup_{1\leq i\leq N_T+1}\|\lambda_T\|_{L^2(I_i;L^2 (\omega))}+\normC{y_d}\Big).
		\end{equation}
	\end{itemize}
\end{lemma}
\begin{proof}
	Since $u_T=\chi_\omega\lambda_T$, Lemma~\ref{lem_y_upper_bound_by_sub_inter} and \eqref{slice-small} give
	\begin{equation}\label{y^T_upper_bounded_sub_inte}
		\|y_T\|_{C([0,T];L^2(0,1))}\leq K_{\mathrm{sl}}\Big(\|y_0\|_{L^2(0,1)}+\sup_{1\leq i\leq N_T+1}\|\lambda_T\|_{L^2(I_i;L^2(\omega))}\Big)\leq 2K_{\mathrm{sl}}R_1<\frac{1}{2\sqrt{2}} .
	\end{equation}
	(i) Fix $t_1\in(0,T]$ and reverse time: for $\tau\in[0,t_1]$ set $\rho(\tau):=\lambda_T(t_1-\tau)$, $\tilde y(\tau):=y_T(t_1-\tau)$ and $\tilde y_d(\tau):=y_d(t_1-\tau)$. By \eqref{maxim}, $\rho$ solves the forward equation $\partial_\tau\rho-\partial_{xx}\rho-\tilde y\,\partial_x\rho=\tilde y_d-\tilde y$ with homogeneous Dirichlet boundary conditions. Multiplying by $\rho$ yields
	\begin{equation}\label{rho_energy}
		\frac{1}{2}\frac{\d}{\d \tau}\norm{\rho}_{L^2(0,1)}^2+\norm{\partial_x\rho}_{L^2(0,1)}^2-\int_0^1\tilde y\,\rho\,\partial_x\rho \d x=\langle\tilde y_d-\tilde y,\rho\rangle .
	\end{equation}
	By the coercivity estimate \eqref{ineq:corecive} with $u=\tilde y(\tau)$, the bound \eqref{y^T_upper_bounded_sub_inte} and the Poincar\'e inequality,
	\begin{equation}\label{uniform_coercive_rho}
		\norm{\partial_x\rho}^2_{L^2(0,1)}-\int_0^1\tilde y\rho\partial_x\rho
		\geq \Big(\pi^2\Big(1-\frac{3\sqrt{2}\,\varepsilon^{4/3}}{2}K_{\mathrm{sl}}R_1\Big)-\frac{\sqrt{2}K_{\mathrm{sl}}R_1}{2\varepsilon^{4}}\Big)\norm{\rho}_{L^2(0,1)}^2
		=:C_{R_1}(\varepsilon)\norm{\rho}_{L^2(0,1)}^2 ,
	\end{equation}
	provided $\varepsilon>0$ is such that the coefficient of $\norm{\partial_x\rho}^2$ is nonnegative. Since $4\sqrt2 K_{\mathrm{sl}}R_1<1$, one has $C_{R_1}(1)\geq\pi^2\left(1-\frac{3}{8}\right)-\frac18>\frac{1}{2}$, while $C_{R_1}(\varepsilon)\to-\infty$ as $\varepsilon\to0^+$; hence there exists $\varepsilon\in(0,1]$ with $C_{R_1}(\varepsilon)=\frac{1}{2}$, which we fix. Then \eqref{rho_energy}, \eqref{uniform_coercive_rho} and Young's inequality $\langle\tilde y_d-\tilde y,\rho\rangle\leq\frac{1}{4}\norm{\rho}^2+\norm{\tilde y-\tilde y_d}^2$ give
	\begin{equation*}
		\frac{\d}{\d \tau}\norm{\rho}_{L^2(0,1)}^2+\frac{1}{2}\norm{\rho}_{L^2(0,1)}^2\leq 2\norm{\tilde y(\tau)-\tilde y_d(\tau)}_{L^2(0,1)}^2 ,
	\end{equation*}
	and Gr\"onwall's inequality on $[0,t_1-t_0]$, written back in the original time variable, yields \eqref{rho_exp_decay}.
	
\smallskip
	\noindent
	(ii) Taking $t_1=T$ in \eqref{rho_exp_decay} and using $\lambda_T(T)=0$, we get, for every $t\in[0,T]$,
	\begin{equation*}
		\|\lambda_T(t)\|_{L^2(0,1)}^2\leq 2\int_{t}^{T} e^{-\frac{s-t}{2}}\|y_T(s)-y_d(s)\|_{L^2(0,1)}^2 \d s
		\leq 2\sum_{k\geq 1}e^{-\frac{(k-1)\bar T}{2}}\,\bar T\,\|y_T-y_d\|_{C([0,T];L^2(0,1))}^2 ,
	\end{equation*}
	where we covered $(t,T)$ by the windows $\big(t+(k-1)\bar T,\,t+k\bar T\big)$, $k\geq1$, of length at most $\bar T$. Hence
	$ \|\lambda_T\|^2_{C([0,T];L^2(0,1))}\leq \frac{2\bar T}{1-e^{-\bar T/2}}\|y_T-y_d\|^2_{C([0,T];L^2(0,1))} $,
	and \eqref{eq_3.20} follows from \eqref{y^T_upper_bounded_sub_inte}, \eqref{def:normC} and $K_{\mathrm{sl}}\geq1$.
\end{proof}

We now turn to the periodic problem. The problem $(Per)_\theta$ admits at least one optimal pair (see, e.g., \cite{B}); any optimal pair $\left(y_{\theta,*}(\cdot),u_{\theta,*}(\cdot)\right)$, together with a $\theta$-periodic adjoint state $\lambda_{\theta,*}(\cdot)\in W(0,\theta)$, solves the following optimality system with $\left(y_{\theta}(\cdot),u_{\theta}(\cdot),\lambda_{\theta}(\cdot)\right)=\left(y_{\theta,*}(\cdot),u_{\theta,*}(\cdot),\lambda_{\theta,*}(\cdot)\right)$ (the derivation is analogous to that of \eqref{maxim}--\eqref{contrT}, see \cite[Theorem 6]{Volkwein2001})
\begin{equation}\label{maximp}
	\left\{
	\begin{aligned}
		&\partial_t y_\theta (t,x) - \partial_{xx} y_\theta (t,x) + y_\theta (t,x)\partial_x y_\theta (t,x) = \chi_\omega \lambda_\theta(t,x) , \\
		&-\partial_t \lambda_\theta (t,x) - \partial_{xx} \lambda_\theta (t,x) - y_\theta (t,x) \partial_x \lambda_\theta (t,x) = - (y_\theta (t,x)-y_d (t,x)), \\
		&y_\theta (t,0)=y_\theta (t,1)=0, \quad \lambda_\theta (t,0)=\lambda_\theta (t,1)=0, \\
		&y_\theta (0,x) = y_\theta (\theta,x), \quad \lambda_\theta (0,x)=\lambda_\theta (\theta,x), 
	\end{aligned}
	\right.
\end{equation}
for $(t,x)\in (0,\theta)\times (0,1)$, 
and
\begin{equation}\label{contrP}
	u_\theta (t,x) = \chi_\omega\lambda_\theta(t,x), \quad \text{a.e. } (t,x)\in (0,\theta)\times (0,1).
\end{equation}
Note that, when $u=0$, the only $\theta$-periodic solution of \eqref{per_state} is $y\equiv0$ (by \eqref{diffineq} and the Poincar\'e inequality, $\|y(\theta)\|_{L^2(0,1)}\leq e^{-\pi^2\theta/2}\|y(0)\|_{L^2(0,1)}$, which forces $y(0)=0$ by periodicity). Consequently $J_\theta(0)=\frac{1}{2}\int_0^\theta\|y_d\|^2_{L^2(0,1)} \d t$ and every optimal control of $(Per)_\theta$ satisfies
\begin{equation}\label{uthetasmall}
	\|u_{\theta,*}\|_{L^2(0,\theta;L^2(\omega))}\leq \sqrt{2J_\theta(0)}\leq \sqrt{\theta}\,\normC{y_d} .
\end{equation}

The following lemma provides uniform estimates for the periodic system. Part (i) is stated for an \emph{arbitrary} admissible periodic pair, since it will be applied in this generality (in Lemma~\ref{lemma_estimate_functional} and in Appendix~\ref{app:B}); part (ii) is stated for the corresponding periodic adjoint state, defined as the (unique) $\theta$-periodic solution of the linear backward equation in \eqref{maximp}, without imposing the feedback relation \eqref{contrP}.
\begin{lemma}\label{Energy}
	Let $\theta>0$ and set $ C_2:=C_2(\theta):= 1+\sqrt{2+1/\theta} $.
	\begin{itemize}
		\item[(i)] Every admissible pair $(y_\theta,u_\theta)$ of \eqref{per_state} satisfies
		\begin{equation}\label{lemma1.3.}
			\|y_\theta\|_{C([0,\theta];L^2(0,1))} + \|y_\theta\|_{L^2(0,\theta;H_0^1 (0,1))}\leq C_2\,\|u_\theta\|_{L^2(0,\theta;L^2 (\omega))} ,
		\end{equation}
		and moreover $y_\theta\in C([0,\theta];H_0^1(0,1))$ with
		\begin{equation}\label{y_p_H01}
			\|y_\theta\|_{C([0,\theta];H_0^1(0,1))}^2 \leq \Big(4+\frac1\theta\Big)\,e^{\frac{8}{\pi}\|u_\theta\|^2_{L^2(0,\theta;L^2(\omega))}}\ \|u_\theta\|^2_{L^2(0,\theta;L^2(\omega))} .
		\end{equation}
		\item[(ii)] Let $R_2>0$ be such that $\sqrt{2}\,C_2R_2\leq\frac12$ and assume $\|u_\theta\|_{L^2(0,\theta;L^2(\omega))}\leq R_2$. Then the backward equation in \eqref{maximp}, with state $y_\theta$, admits a unique $\theta$-periodic solution $\lambda_\theta$, and there exists $C_3=C_3(\theta,R_2)>0$ such that
		\begin{equation}\label{lemma1.3}
			\|\lambda_\theta\|_{C([0,\theta];L^2(0,1))} + \|\lambda_\theta\|_{L^2(0,\theta;H_0^1 (0,1))}\leq C_3\,\|y_\theta-y_d\|_{L^2(0,\theta;L^2 (0,1))} .
		\end{equation}
	\end{itemize}
	In view of \eqref{uthetasmall}, the assumption of (ii) holds for every optimal pair of $(Per)_\theta$ as soon as $\sqrt{\theta}\,\normC{y_d}\leq R_2$ (and $\sqrt2 C_2R_2\leq\frac12$).
\end{lemma}
\begin{proof}
	(i) Set $E:=\|u_\theta\|^2_{L^2(0,\theta;L^2(\omega))}$. As in the proof of Lemma~\ref{energyT}, the pair $(y_\theta,u_\theta)$ satisfies \eqref{diffineq} on $(0,\theta)$. Integrating \eqref{diffineq} over $(0,\theta)$ and using $\|y_\theta(\theta)\|_{L^2}=\|y_\theta(0)\|_{L^2}$ gives
	\begin{equation}\label{H01}
		\int_0^\theta \norm{y_\theta(t)}^2_{H_0^1(0,1)}\d t\leq E .
	\end{equation}
	Integrating \eqref{diffineq} over $(s,\theta)$ for $s\in(0,\theta)$, averaging in $s$ over $(0,\theta)$ and using the periodicity, the Poincar\'e inequality and \eqref{H01}, we obtain
	\begin{equation*}
		\|y_\theta(0)\|^2_{L^2(0,1)}=\|y_\theta(\theta)\|^2_{L^2(0,1)}\leq \frac1\theta\int_0^\theta\|y_\theta(s)\|^2_{L^2(0,1)} \d s+E\leq \Big(\frac{1}{\theta}+1\Big)E ,
	\end{equation*}
	whence, integrating \eqref{diffineq} over $(0,t)$, $\sup_{t\in[0,\theta]}\|y_\theta(t)\|^2_{L^2(0,1)}\leq \|y_\theta(0)\|^2_{L^2}+E\leq(2+\frac1\theta)E$. Together with \eqref{H01} this proves \eqref{lemma1.3.}.
	
	We next prove \eqref{y_p_H01}; the computations below are justified by a standard Galerkin approximation. Multiplying the equation by $-\partial_{xx}y_\theta$ and using the Agmon and Poincar\'e inequalities $\norm{y_\theta}_{L^\infty(0,1)}^2\leq2\norm{y_\theta}_{L^2}\norm{\partial_x y_\theta}_{L^2}\leq \frac{2}{\pi}\norm{\partial_x y_\theta}^2_{L^2}$, we get
	\begin{align*}
		\frac{1}{2}\frac{\d}{\d t}\norm{\partial_x y_\theta}_{L^2}^2+\norm{\partial_{xx}y_\theta}_{L^2}^2
		&\leq \norm{y_\theta}_{L^\infty}\norm{\partial_x y_\theta}_{L^2}\norm{\partial_{xx}y_\theta}_{L^2}+\norm{u_\theta}_{L^2(\omega)}\norm{\partial_{xx}y_\theta}_{L^2}\\
		&\leq \frac{1}{2}\norm{\partial_{xx}y_\theta}_{L^2}^2+\frac{2}{\pi}\norm{\partial_x y_\theta}_{L^2}^4+\norm{u_\theta}_{L^2(\omega)}^2 ,
	\end{align*}
	so that $\frac{\d}{\d t}\norm{\partial_x y_\theta}^2_{L^2}\leq g(t)\norm{\partial_x y_\theta}^2_{L^2}+2\norm{u_\theta}^2_{L^2(\omega)}$ with $g:=\frac{4}{\pi}\norm{\partial_x y_\theta}^2_{L^2}$ and, by \eqref{H01}, $\int_0^\theta g\leq \frac{4}{\pi}E=:G$. Gr\"onwall's inequality between $s$ and $\theta$ ($0\leq s\leq\theta$) gives $\norm{\partial_x y_\theta(\theta)}^2_{L^2}\leq e^{G}\big(\norm{\partial_x y_\theta(s)}^2_{L^2}+2E\big)$; averaging in $s$ over $(0,\theta)$ and using \eqref{H01} and the periodicity,
	\begin{equation*}
		\norm{\partial_x y_\theta(0)}^2_{L^2(0,1)}=\norm{\partial_x y_\theta(\theta)}^2_{L^2(0,1)}\leq e^{G}\Big(\frac1\theta+2\Big)E .
	\end{equation*}
	Applying Gr\"onwall's inequality once more, this time between $0$ and $t\in[0,\theta]$, yields
	\begin{equation*}
		\sup_{t\in[0,\theta]}\norm{\partial_x y_\theta(t)}^2_{L^2}\leq e^{G}\big(\norm{\partial_x y_\theta(0)}^2_{L^2}+2E\big)\leq e^{2G}\big(4+\tfrac1\theta\big)E ,
	\end{equation*}
	which is \eqref{y_p_H01}. Once this bound is known, $y_\theta\partial_x y_\theta\in L^2(0,\theta;L^2(0,1))$, and parabolic regularity combined with periodicity yields $y_\theta\in C([0,\theta];H_0^1(0,1))$.
	
	\smallskip
	\noindent
	(ii) By \eqref{lemma1.3.}, $\|y_\theta\|_{C([0,\theta];L^2(0,1))}\leq C_2R_2$ with $\sqrt2 C_2R_2\leq\frac12$. Arguing as in the proof of Lemma~\ref{regular_lambda}(i) (time reversal $\rho(\tau):=\lambda_\theta(\theta-\tau)$, energy identity, coercivity estimate \eqref{ineq:corecive} with $u=y_\theta$ and $\varepsilon=1$, Poincar\'e inequality), and keeping this time half of the gradient term, we obtain
	\begin{equation}\label{rho_periodic}
		\frac{\d}{\d \tau}\norm{\rho}_{L^2(0,1)}^2+\kappa\norm{\partial_x\rho}^2_{L^2(0,1)}+C_{R_2}\norm{\rho}_{L^2(0,1)}^2\leq \frac{1}{C_{R_2}}\norm{(y_\theta-y_d)(\theta-\tau)}_{L^2(0,1)}^2 ,
	\end{equation}
	with $\kappa:=1-\frac{3\sqrt2}{4}C_2R_2$ and $C_{R_2}:=\frac{\pi^2\kappa}{2}-\frac{\sqrt2 C_2R_2}{4}$; since $\sqrt2 C_2R_2\leq\frac12$, one has $\kappa\geq\frac58$ and $C_{R_2}\geq\frac{5\pi^2}{16}-\frac18>1$.
	The linear backward equation being considered on the space of $\theta$-periodic functions, \eqref{rho_periodic} (with zero right-hand side for the difference of two solutions) shows that the affine map $\lambda_\theta(\theta)\mapsto\lambda_\theta(0)$ is a contraction of $L^2(0,1)$ with constant $e^{-C_{R_2}\theta/2}<1$; hence the periodic solution $\lambda_\theta$ exists and is unique. Gr\"onwall's inequality applied to \eqref{rho_periodic} over one period, combined with the periodicity $\rho(\theta)=\rho(0)$, gives
$(1-e^{-C_{R_2}\theta})\|\lambda_\theta(\theta)\|^2_{L^2(0,1)}\leq \frac{1}{C_{R_2}}\|y_\theta-y_d\|^2_{L^2(0,\theta;L^2(0,1))}$,
then, for every $\tau\in[0,\theta]$, $\norm{\rho(\tau)}^2_{L^2}\leq\norm{\rho(0)}^2_{L^2}+\frac{1}{C_{R_2}}\|y_\theta-y_d\|^2_{L^2L^2}$, while integrating \eqref{rho_periodic} over $(0,\theta)$ gives $\kappa\int_0^\theta\norm{\partial_x\lambda_\theta}^2_{L^2}\leq \frac{1}{C_{R_2}}\|y_\theta-y_d\|^2_{L^2L^2}$. Combining these three estimates yields \eqref{lemma1.3} with an explicit constant $C_3(\theta,R_2)$.
\end{proof}

We also claim that the periodic state is unique on the small control ball used below. Indeed, if $y_1$ and $y_2$ are two periodic states associated with the same control $u_\theta$ and $z:=y_1-y_2$, then
$\partial_tz-\partial_{xx}z+\frac12\partial_x((y_1+y_2)z)=0$.
If $\|u_\theta\|_{L^2(0,\theta;L^2(\omega))}\leq R_2$ and $\sqrt2 C_2R_2\leq\frac12$, Lemma~\ref{Energy}(i) gives $\|y_j\|_{C([0,\theta];L^2)}\leq C_2R_2$ for $j=1,2$. Multiplication by $z$, followed by \eqref{ineq:corecive} and integration over one period, then yields $z=0$. Thus the periodic control-to-state map, and hence the reduced functional $J_\theta$, is single-valued on this ball.

\begin{remark}\label{rem:theta_constants}
	The constant $C_2(\theta)$ is decreasing and the admissible threshold for $R_2$, namely
	$\frac{1}{2\sqrt2 C_2(\theta)}=\frac{1}{2\sqrt2+2\sqrt{4+2/\theta}}$, is increasing with respect to $\theta$; both degenerate as $\theta\to0^+$. As $\theta\to+\infty$, $C_2(\theta)\to1+\sqrt2$ and this particular threshold stays bounded away from zero. Other constants, notably those furnished by the periodic Riccati theory, need not be uniform in $\theta$, and we make no claim of a $\theta$-uniform decay rate $\mu$.
\end{remark}

\subsection{Uniform bounds for optimal solutions}\label{sec:uniform_bound_optima}
We give a uniform (with respect to $T$) small-data bound for the optimal solutions $\left(y_{T,*},u_{T,*}\right)$ of $(OCP)_T$ in the next lemma. This is the only regime in which the result is used below. Its proof, given in Appendix~\ref{app:A}, relies only on local exact controllability to sufficiently regular trajectories.
\begin{lemma}\label{uniform_bound_optima}
	There exist constants $r_{\mathrm{ub}}>0$ and $K_{\mathrm{ub}}>0$, depending only on $\omega$, such that the following holds. In this lemma, take $\bar T=4$ in \eqref{def:slices}, $K_{\mathrm{sl}}=K_{\mathrm{sl}}(4)$ in \eqref{def:K}, and $K_1=K_1(4)$ in Lemma~\ref{regular_lambda}. For all $T>0$, $y_0\in L^2(0,1)$ and $\theta$-periodic $y_d\in C([0,\infty);L^2(0,1))$ satisfying
	\begin{equation}\label{ub_small_data}
		\|y_0\|_{L^2(0,1)}+\normC{y_d}\leq r_{\mathrm{ub}},
	\end{equation}
	every optimal triple $\left(y_{T,*},u_{T,*},\lambda_{T,*}\right)$ of $(OCP)_T$ satisfies
	\begin{equation}\label{uniform_bound_optima1}
		\max_{1\leq i\leq N_T+1}\ \int_{I_i}\left(\|u_{T,*}(t)\|^2_{L^2(\omega)}+\|y_{T,*}(t)\|^2_{L^2(0,1)}\right)\d t \leq K_{\mathrm{ub}}\left(\|y_0\|^2_{L^2(0,1)}+\normC{y_d}^2\right),
	\end{equation}
	where $N_T=N_T(4)$ and $I_i=I_i(4)$ are defined by \eqref{def:slices}. As a consequence:
	\begin{itemize}
		\item[(i)] setting $K_{\mathrm{ub},1}:=K_{\mathrm{sl}} (1+\sqrt{K_{\mathrm{ub}}})$, we have
		\begin{equation*}\label{uniform_bound_y^T}
			\|y_{T,*}\|_{C([0,T];L^2(0,1))}\leq K_{\mathrm{ub},1}\left(\|y_0\|_{L^2(0,1)}+\normC{y_d}\right);
		\end{equation*}
		\item[(ii)] setting $L_1:=\frac{1}{4\sqrt{2}K_{\mathrm{sl}}}$ and $\tilde K_1:=\frac{1}{2\sqrt{K_{\mathrm{ub}}}}$, for every $\tilde R\in(0,L_1)$, if moreover
		\begin{equation}\label{eq_3.19}
			\|y_0\|_{L^2(0,1)}\leq\tilde R\qquad\text{and}\qquad \|y_0\|_{L^2(0,1)}+\normC{y_d}\leq\tilde K_1\,\tilde R,
		\end{equation}
		then, setting $K_{1,2}:=K_1 (1+\sqrt{K_{\mathrm{ub}}})$,
		\begin{equation}\label{uniform_bound_u^T}
			\|u_{T,*}\|_{L^\infty(0,T;L^2(\omega))}\leq K_{1,2}\left(\|y_0\|_{L^2(0,1)}+\normC{y_d}\right).
		\end{equation}
	\end{itemize}
\end{lemma}
\begin{proof}
	The slice bound \eqref{uniform_bound_optima1} is exactly Lemma~\ref{prop4.1} in Appendix~\ref{app:A}. Since $u_{T,*}=\chi_\omega\lambda_{T,*}$, it implies
	\begin{equation}\label{eq_3.20bis}
		\sup_{1\leq i\leq N_T+1}\|\lambda_{T,*}\|_{L^2(I_i;L^2(\omega))}\leq \sqrt{K_{\mathrm{ub}}}\left(\|y_0\|_{L^2(0,1)}+\normC{y_d}\right).
	\end{equation}
	Then (i) follows from Lemma~\ref{lem_y_upper_bound_by_sub_inter} (with $\bar T=4$) and \eqref{eq_3.20bis}. As for (ii): under \eqref{eq_3.19}, the estimate \eqref{eq_3.20bis} gives $\sup_i\|\lambda_{T,*}\|_{L^2(I_i;L^2(\omega))}\leq\sqrt{K_{\mathrm{ub}}}\,\tilde K_1\tilde R=\tilde R/2<\tilde R$, so that the assumptions \eqref{slice-small} of Lemma~\ref{regular_lambda} are satisfied with $R_1:=\tilde R<L_1=\frac{1}{4\sqrt2 K_{\mathrm{sl}}}$. Hence, by \eqref{eq_3.20} and \eqref{eq_3.20bis},
	\begin{multline*}
		\|\lambda_{T,*}\|_{C([0,T];L^2(0,1))}\leq K_1\left(\|y_0\|_{L^2(0,1)}+\sqrt{K_{\mathrm{ub}}}\left(\|y_0\|_{L^2(0,1)}+\normC{y_d}\right)+\normC{y_d}\right)\\
		\leq K_{1,2}\left(\|y_0\|_{L^2(0,1)}+\normC{y_d}\right),
	\end{multline*}
	and \eqref{uniform_bound_u^T} follows from $u_{T,*}=\chi_\omega\lambda_{T,*}$.
\end{proof}

\subsection{\texorpdfstring{Strict convexity and local uniqueness for $(OCP)_T$ and $(Per)_\theta$}{Strict convexity and local uniqueness}}\label{sec:local_uniqueness}
Denote by 
$$
J_T(u(\cdot)):=\frac{1}{2}\int_0^T \big(\| y(t)-y_d(t) \|_{L^2(0,1)}^2+ \|u(t)\|_{L^2(\omega)}^2\big)\d t 
$$
the reduced functional for $ (OCP)_T $, where $(y(\cdot),u(\cdot))\in L^2(0,T;L^2(0,1))\times L^2(0,T;L^2(\omega))$ satisfies $\eqref{state}$, and by 
$$ 
J_\theta(u(\cdot)):=\frac{1}{2}\int_0^\theta \big(\| y(t)-y_d(t) \|_{L^2(0,1)}^2+ \|u(t)\|_{L^2(\omega)}^2\big)\d t 
$$ 
the reduced functional for $ (Per)_\theta $, where $(y(\cdot),u(\cdot))\in L^2(0,\theta;L^2(0,1)) \times L^2(0,\theta;L^2(\omega))$ is the solution of $ \eqref{per_state} $.

For any control $u_T\in L^\infty(0,T;L^2(\omega))$, let $y_T$ be the corresponding state solving \eqref{state}, and let $\lambda_T$ solve the adjoint equation in \eqref{maxim} with this state $y_T$, terminal condition $\lambda_T(T)=0$, but without imposing the relation \eqref{contrT}. Then the second Gateaux derivative\footnote{Recall that the first  Gateaux derivative of $J_T$ at $u_T(\cdot)$ in the direction $v(\cdot)$ is  $J_T^{\prime}\left(u_T(\cdot)\right)[v(\cdot)]:=\lim\limits_{\alpha\to 0}\frac{J_T\left(u_T(\cdot)+\alpha v(\cdot)\right) - J_T\left(u_T(\cdot)\right)}{\alpha}$, and that the second Gateaux derivative of $ J_T $ at $ u_T(\cdot) $ in the direction $ \left(v_1(\cdot),v_2(\cdot)\right) $ is $J_T^{\prime\prime}(u_T )[v_1 ,v_2 ]:= \lim\limits_{\alpha\to 0}\frac{J_T^{\prime}\left(u_T(\cdot)+\alpha v_2(\cdot)\right)[v_1(\cdot)]- J_T^{\prime}\left(u_T(\cdot)\right)[v_1(\cdot)]}{\alpha}$. } of $ J_T $ at $ u_T(\cdot) $ in the direction $ (v_T(\cdot),v_T(\cdot)) $, denoted by $ J_T^{\prime\prime}(u_T )[v_T ,v_T ]$, is (see \cite[Theorem 5.2]{CT})
\begin{equation}\label{eq_J_T_sec_der}
	J_T^{\prime\prime}(u_T )[v_T ,v_T ] =  - \int_{0}^{T}\langle z_T(t)\partial_x \lambda_T(t), z_T(t)\rangle\d t +  \int_{0}^{T} \|z_T(t)\|_{L^2(0,1)}^2\d t +  \int_{0}^{T}\|v_T(t)\|_{L^2(\omega)}^2\d t,
\end{equation}
where $ \left(z_T(\cdot), v_T(\cdot)\right) $ satisfies 
\begin{equation}\label{linearizedT}
	\left\{
	\begin{aligned}
		&\partial_t z_T  - \partial_{xx} z_T + \partial_x \left(y_T z_T\right)  = \chi_\omega v_T  , \\
		&z_T (t,0)=z_T (t,1)=0,\\
		&z_T (0,x) = 0,
	\end{aligned}
	\right.
\end{equation}
for $(t,x)\in (0,T)\times (0,1)$.
Similarly, let $u_\theta\in L^2(0,\theta;L^2(\omega))$ satisfy the smallness condition of Lemma~\ref{Energy}(ii), let $y_\theta$ be its unique periodic state solving \eqref{per_state}, and let $\lambda_\theta$ be the periodic adjoint state given by that lemma, without imposing \eqref{contrP}. Then the second Gateaux derivative of $ J_\theta $ at $ u_\theta(\cdot) $ in the direction $ (v_\theta(\cdot),v_\theta(\cdot)) $ is
\begin{equation}\label{eq_J_theta_sec_der}
	J_{\theta}^{\prime\prime}(u_\theta)[v_\theta,v_\theta] =  - \int_{0}^{\theta}\langle z_\theta(t)\partial_x \lambda_\theta(t), z_\theta(t)\rangle\d t +  \int_{0}^{\theta} \|z_\theta(t)\|_{L^2(0,1)}^2\d t +  \int_{0}^{\theta}\|v_\theta(t)\|_{L^2(\omega)}^2\d t,
\end{equation}
where $ \left(z_\theta(\cdot), v_\theta(\cdot)\right) $ satisfies  
\begin{equation}\label{linearizedP}
	\left\{
	\begin{aligned}
		&\partial_t z_\theta  - \partial_{xx} z_\theta  + \partial_x \left(y_\theta z_\theta\right)  = \chi_\omega v_\theta  , \\
		&z_\theta (t,0)=z_\theta (t,1)=0, \\
		&z_\theta (0,x) = z_\theta (\theta,x),
	\end{aligned}
	\right.
\end{equation}
for $(t,x)\in (0,\theta)\times (0,1)$.

The key convexity property is the following. Compared with a statement tailored to optimal controls only, it is stated on balls with a \emph{free radius} $r$; this flexibility is used in Appendix~\ref{app:C}, where we need the strict convexity on a ball containing both the minimizers of $(OCP)_T$ and the solution of the optimality system constructed there by a fixed point argument. The proof is given in Appendix~\ref{app:B}.
\begin{lemma}\phantomsection\label{lem:convexity}
	\begin{itemize}
		\item[(i)] There exists $r_0\in(0,1]$ such that, for all $T>0$, $r\in(0,r_0]$, and all $y_0\in L^2(0,1)$, $\theta$-periodic $y_d\in C([0,\infty);L^2(0,1))$ with $\|y_0\|_{L^2(0,1)}\leq r_0$ and $\normC{y_d}\leq r_0$, one has
		\begin{equation}\label{convexT}
			J_T^{\prime\prime}(u)[v,v] \geq \frac{1}{2}\int_0^T\|v(t)\|^2_{L^2(\omega)}\d t
			\qquad\forall v\in L^2(0,T;L^2(\omega)),
		\end{equation}
		for every $u$ in the ball $B_r:=\left\{u\in L^\infty(0,T;L^2(\omega))\,\mid\, \|u\|_{L^\infty(0,T;L^2(\omega))}\leq r\right\}$. In particular $J_T$ is strictly convex on the convex set $B_r$.
		\item[(ii)] For every $\theta>0$ there exists $r_\theta\in(0,1]$ such that, if $\normC{y_d}\leq r_\theta$, then
		\begin{equation}\label{convexP}
			J_\theta^{\prime\prime}(u)[v,v] \geq \frac{1}{2}\int_0^\theta\|v(t)\|^2_{L^2(\omega)}\d t
			\qquad\forall v\in L^2(0,\theta;L^2(\omega)),
		\end{equation}
		for every $u$ in the ball $B^{per}_{r_\theta}:=\{u\in L^2(0,\theta;L^2(\omega))\,\mid\, \|u\|_{L^2(0,\theta;L^2(\omega))}\leq r_\theta \}$. In particular $J_\theta$ is strictly convex on the convex set $B^{per}_{r_\theta}$.
	\end{itemize}
\end{lemma}

\begin{proposition}\label{local_uniqueness}
	Let $T>0$, $y_0\in L^2(0,1)$ and $\theta$-periodic $y_d\in C([0, \infty);L^2 (0,1))$.
	There exists $ \varepsilon>0 $ (independent of $T$) such that, if $\max(\|y_0\|_{L^2(0,1)}, \normC{y_d}) \leq \varepsilon$, then the problem $ (OCP)_T $ has a unique optimal triple $ \left(y_{T,*}(\cdot),u_{T,*}(\cdot),\lambda_{T,*}(\cdot)\right) $.
	Similarly, if $ \normC{y_d} \leq \varepsilon$, then the problem $ (Per)_\theta $ has a unique optimal triple $ \left(y_{\theta,*}(\cdot),u_{\theta,*}(\cdot),\lambda_{\theta,*}(\cdot)\right) $.
\end{proposition}
\begin{proof}
	We take
	\begin{equation}\label{threshold1}
		\varepsilon := \min\Big( r_\theta, \frac{r_\theta}{\sqrt{\theta}}, \frac{r_{\mathrm{ub}}}{2}, \tilde R, \tfrac12\tilde K_1\tilde R, r_0, \frac{r_0}{2K_{1,2}}\Big) ,
	\end{equation}
	where $r_\theta$ and $r_0$ are given by Lemma~\ref{lem:convexity}, $\tilde R:=\frac{L_1}{2}$, and $r_{\mathrm{ub}}$, $L_1$, $\tilde K_1$, $K_{1,2}$ are the constants of Lemma~\ref{uniform_bound_optima}.
	
	We first consider $(Per)_\theta$. By \eqref{uthetasmall}, every optimal control satisfies $\|u_{\theta,*}\|_{L^2(0,\theta;L^2(\omega))}\leq\sqrt\theta\,\normC{y_d}\leq r_\theta$, hence lies in the convex set $B^{per}_{r_\theta}$, on which $J_\theta$ is strictly convex by Lemma~\ref{lem:convexity}(ii). If $u_1\neq u_2$ were two optimal controls, then $\frac{u_1+u_2}{2}\in B^{per}_{r_\theta}$ and, by strict convexity, $J_\theta\big(\frac{u_1+u_2}{2}\big)<\frac{1}{2} J_\theta(u_1)+\frac{1}{2} J_\theta(u_2)=\inf J_\theta$, a contradiction. Hence the optimal control is unique; the optimal state and (by Lemma~\ref{Energy}(ii)) the periodic adjoint state are then unique as well.
	
	We now consider $(OCP)_T$. Since $2\varepsilon\leq r_{\mathrm{ub}}$, the small-data assumption \eqref{ub_small_data} holds. Moreover, $\tilde R=\frac{L_1}{2}<L_1$ and $\varepsilon\leq\min\{\tilde R,\frac{1}{2}\tilde K_1\tilde R\}$, so the assumptions \eqref{eq_3.19} hold. Therefore, by \eqref{uniform_bound_u^T}, every optimal control of $(OCP)_T$ satisfies $\|u_{T,*}\|_{L^\infty(0,T;L^2(\omega))}\leq K_{1,2}\left(\|y_0\|_{L^2(0,1)}+\normC{y_d}\right)\leq 2K_{1,2}\,\varepsilon\leq r_0$; hence all optimal controls lie in the convex set $B_{r_0}$, on which $J_T$ is strictly convex by Lemma~\ref{lem:convexity}(i). Arguing as above, the optimal control is unique, and so are the optimal state and the adjoint state.
\end{proof}

\section{Global periodic turnpike}\label{sec:periodicturnpike}
\subsection{The local turnpike result}\label{sec:local_turnpike}
We first establish a local turnpike result. Its proof consists in studying the perturbation of \eqref{maxim}--\eqref{contrT} with respect to \eqref{maximp}--\eqref{contrP}, by means of the periodic Riccati theory developed in \cite{TZZ}, of a fixed point argument, and of the convexity Lemma~\ref{lem:convexity}, which serves to identify the solution of the optimality system constructed by the fixed point argument with the (unique) minimizer of $(OCP)_T$. The detailed proof is given in Appendix~\ref{app:C}.
\begin{lemma}\label{lem:periodic_regularity}
	Fix $\alpha\in(3/4,1)$ and set $A_0=-\partial_{xx}$, with domain $D(A_0)=H^2(0,1)\cap H_0^1(0,1)$. There exist $r_\alpha>0$ and $C_\alpha>0$ such that, if $\normC{y_d}\leq r_\alpha$, the optimal periodic triple satisfies
	\begin{equation}\label{periodic_regularity}
		\|y_{\theta,*}\|_{C([0,\theta];H_0^1(0,1))} + \|\lambda_{\theta,*}\|_{C([0,\theta];H_0^1(0,1))} + \|\lambda_{\theta,*}\|_{L^\infty(0,\theta;D(A_0^\alpha))} \leq C_\alpha\normC{y_d}.
	\end{equation}
	Consequently,
	\begin{equation}\label{periodic_lambdax}
		\|\partial_x\lambda_{\theta,*}\|_{L^\infty((0,\theta)\times(0,1))} \leq C_\alpha\normC{y_d}.
	\end{equation}
\end{lemma}
\begin{proof}
Write $\delta:=\normC{y_d}$. By \eqref{uthetasmall} and \eqref{y_p_H01}, after fixing a sufficiently small upper bound for $\delta$, we have
\begin{equation}\label{periodic_yH1}
	\|y_{\theta,*}\|_{C([0,\theta];H_0^1(0,1))}\leq C\delta.
\end{equation}
Lemma~\ref{Energy}(ii) also yields the estimate $\|\lambda_{\theta,*}\|_{C([0,\theta];L^2(0,1))} + \|\lambda_{\theta,*}\|_{L^2(0,\theta;H_0^1(0,1))}\leq C\delta$. Reversing time and setting $\rho(t):=\lambda_{\theta,*}(\theta-t)$, $\widetilde y(t):=y_{\theta,*}(\theta-t)$ and $\widetilde y_d(t):=y_d(\theta-t)$, $\rho$ is periodic and
\begin{equation}\label{periodic_heat}
\partial_t\rho+A_0\rho=F, \qquad F:=\widetilde y\,\partial_x\rho+\widetilde y_d-\widetilde y.
\end{equation}
The above estimates and the embedding $H_0^1(0,1)\hookrightarrow L^\infty(0,1)$ yield
$\|F\|_{L^2(0,\theta;L^2(0,1))}\leq C\delta$.
The periodic maximal-$L^2$ estimate for $\partial_t+A_0$ reads
$$
\|\partial_t\rho\|_{L^2(0,\theta;L^2)} + \|A_0\rho\|_{L^2(0,\theta;L^2)} + \|\rho\|_{C([0,\theta];D(A_0^{1/2}))} \leq C_\theta\|F\|_{L^2(0,\theta;L^2)}.
$$
This indeed follows by expanding the periodic functions in the temporal Fourier basis and in the Dirichlet eigenbasis of $A_0$: each coefficient is divided by $2\pi i n/\theta+\pi^2k^2$, and the claimed $H^1$-in-time and $D(A_0)$ bounds follow by Parseval's identity; the continuous trace in $D(A_0^{1/2})=H_0^1(0,1)$ is the standard Hilbert-space energy trace estimate. Hence $\|\rho\|_{C([0,\theta];H_0^1)}\leq C\delta$. Returning to \eqref{periodic_heat}, we now have $F\in L^\infty(0,\theta;L^2(0,1))$ and $\|F\|_{L^\infty L^2}\leq C\delta$.

Extend $F$ and $\rho$ periodically to $\mathbb R$. The variation-of-constants formula over one complete period gives, for every $t$,
\begin{equation}\label{periodic_mild}
\rho(t)=(I-e^{-\theta A_0})^{-1}\int_0^\theta e^{-sA_0}F(t-s) \d s.
\end{equation}
Since $A_0\geq\pi^2 I$, the inverse in \eqref{periodic_mild} is bounded on $L^2(0,1)$; moreover $\|A_0^\alpha e^{-sA_0}\|_{\mathcal L(L^2)}\leq C_\alpha s^{-\alpha}$ for $s>0$. Because $\alpha<1$, \eqref{periodic_mild} yields $\|\rho\|_{L^\infty(0,\theta;D(A_0^\alpha))}\leq C_\alpha\delta$. This proves \eqref{periodic_regularity}. Finally, $D(A_0^\alpha)\hookrightarrow H^{2\alpha}(0,1)\hookrightarrow W^{1,\infty}(0,1)$ because $2\alpha>3/2$, and \eqref{periodic_lambdax} follows.
\end{proof}
\begin{proposition}\label{local_turnpike} 
There exist $M_0>0$ and $\mu>0$, depending only on $\theta$ and $\omega$, with the following property. For every $M\in(0,M_0]$, there exists $\varepsilon_M>0$, depending only on $\theta$, $\omega$ and $M$, such that, for every $T>0$, every $y_0\in H_0^1(0,1)$ and every $\theta$-periodic $y_d\in C([0,\infty);L^2(0,1))$ satisfying
	\begin{equation}\label{smallness_target_datum_varepsilon}
		\|y_0\|_{H_0^1(0,1)}+ \normC{y_d} \leq \varepsilon_M,
	\end{equation} 
	then $(OCP)_T$ has a unique optimal triple $ \left(y_{T,*}(\cdot),u_{T,*}(\cdot),\lambda_{T,*}(\cdot)\right)$ satisfying \eqref{maxim}--\eqref{contrT}, and it has the exponential periodic turnpike property
	\begin{multline}\label{turnpike}
			\|y_{T,*}(t) - y_{\theta,*}(t)\|_{L^2(0,1)}+\|\lambda_{T,*}(t)-\lambda_{\theta,*}(t)\|_{L^2(0,1)} + \|u_{T,*}(t) - u_{\theta,*}(t)\|_{L^2(\omega)} \\\leq M\big(e^{-\mu t }+e^{-\mu(T-t)}\big),\quad \text{a.e. }  t \in[0,T],
	\end{multline}
	where $ \left(y_{\theta,*}(\cdot),u_{\theta,*}(\cdot),\lambda_{\theta,*}(\cdot)\right) $, the unique optimal triple of $(Per)_\theta$, is extended by $ \theta $-periodicity over the whole real line. The same constants $M_0$, $\mu$ and $\varepsilon_M$ apply when $y_d$ is replaced by any phase shift $y_d(\cdot+s)$, $s\in\mathbb R$.
\end{proposition}

\subsection{Two cost comparison lemmas}\label{sec:global_attractor}
Let $T>0$, let $y_0\in L^2(0,1)$, and let $y_d\in C([0,\infty);L^2(0,1))$ be $\theta$-periodic. 
We first estimate the difference between the value of $J_T(\cdot)$ at a control $ u $ and $\frac{T}{\theta}$ times the value of $J_\theta(\cdot)$ at a periodic control $ \bar{u} $, assuming that $ u $ and $ \bar{u} $ satisfy a turnpike-like estimate.
\begin{lemma}\label{lemma_estimate_functional}
	Let $\bar{u}\in L^2(0,\theta;L^2 (\omega))$ be a control and  $\bar{y}$ be the corresponding solution to \eqref{per_state}, both extended by $ \theta $-periodicity over the whole real line. Let $u\in L^2(0,T;L^2 (\omega))$ be a control and $y$ be the corresponding solution to \eqref{state}. Suppose that
	\begin{equation}\label{lemma_estimate_functional_turnpike_estimate}
			\|y(t) - \bar{y}(t)\|_{L^2(0,1)} + \|u(t) - \bar{u}(t)\|_{L^2(\omega)} \leq M\big(e^{-\mu t }+e^{-\mu(T-t)}\big),\quad \text{a.e. } t \in[0,T],
	\end{equation}for some $M>0$ and $\mu>0$.
	Then,
	\begin{equation}\label{lemma_estimate_functional_conclusion}
		\Big| J_{T}(u)-\frac{T}{\theta}J_\theta(\bar{u})\Big|
		\leq C_4\big(1+\|\bar{u}\|_{ L^2(0,\theta;L^2 (\omega))}+\|\bar{u}\|_{ L^2(0,\theta;L^2 (\omega))}^2+\normC{y_d}+\normC{y_d}^2\big),
	\end{equation}
	where
	\begin{equation}\label{C4}
		C_4:=\frac{2M^2}{\mu}+c_\mu\big(1+\sqrt\theta\, C_2\big)+1+2\theta C_2^2,
		\qquad
		c_\mu:=\frac{2M}{\sqrt{2\mu}\left(1-e^{-\mu\theta}\right)},
	\end{equation}
	depends only on $M,\mu,\theta$ and is independent of $T$.
\end{lemma}
\begin{proof}
	Set $g(t):=M\left(e^{-\mu t}+e^{-\mu(T-t)}\right)$, so that
	\begin{equation}\label{gL2}
		\|g\|^2_{L^2(0,T)}\leq 2M^2\int_0^T\big(e^{-2\mu t}+e^{-2\mu(T-t)}\big)\d t\leq \frac{2M^2}{\mu}.
	\end{equation}
	Moreover, for every nonnegative $\theta$-periodic $\varphi\in L^2_{loc}([0,\infty))$, covering $(0,T)$ by periods and using the Cauchy--Schwarz inequality,
	\begin{equation*}
		\int_0^T e^{-\mu t}\varphi(t) \d t\leq \sum_{k\geq0}e^{-\mu k\theta}\int_0^\theta e^{-\mu s}\varphi(s) \d s\leq \frac{\|\varphi\|_{L^2(0,\theta)}}{\sqrt{2\mu}\left(1-e^{-\mu\theta}\right)} ;
	\end{equation*}
	the same bound holds for $\int_0^Te^{-\mu(T-t)}\varphi(t) \d t$ (change of variable $t\mapsto T-t$, the function $\varphi(T-\cdot)$ being $\theta$-periodic as well). Hence
	\begin{equation}\label{gphi}
		\int_0^T g(t)\varphi(t) \d t\leq c_\mu \|\varphi\|_{L^2(0,\theta)} .
	\end{equation}
	Finally, writing $\int_0^T h=\lfloor T/\theta\rfloor\int_0^\theta h+\int_{\lfloor T/\theta\rfloor\theta}^{T}h$ for a nonnegative $\theta$-periodic $h\in L^1_{loc}$,
	\begin{equation}\label{periodleftover}
		\Big|\int_0^T h(t) \d t-\frac{T}{\theta}\int_0^\theta h(t) \d t\Big|\leq 2\int_0^\theta h(t) \d t .
	\end{equation}
	Now, $2J_T(u)-\frac{2T}{\theta}J_\theta(\bar u)=A_u+A_y$ with
	$A_u:=\int_0^T\|u\|^2_{L^2(\omega)}-\frac{T}{\theta}\int_0^\theta\|\bar u\|^2_{L^2(\omega)}$ and
	$A_y:=\int_0^T\|y-y_d\|^2_{L^2(0,1)}-\frac{T}{\theta}\int_0^\theta\|\bar y-y_d\|^2_{L^2(0,1)}$.
	Expanding $\|u\|^2=\|u-\bar u\|^2+2\langle u-\bar u,\bar u\rangle+\|\bar u\|^2$ and applying \eqref{lemma_estimate_functional_turnpike_estimate}, \eqref{gL2}, \eqref{gphi} with $\varphi(t)=\|\bar u(t)\|_{L^2(\omega)}$, and \eqref{periodleftover} with $h(t)=\|\bar u(t)\|^2_{L^2(\omega)}$,
	\begin{equation*}
		|A_u|\leq \|g\|^2_{L^2(0,T)}+2\int_0^Tg\,\varphi+2\|\bar u\|^2_{L^2(0,\theta;L^2(\omega))}
		\leq \frac{2M^2}{\mu}+2c_\mu\|\bar u\|_{L^2(0,\theta;L^2(\omega))}+2\|\bar u\|^2_{L^2(0,\theta;L^2(\omega))} .
	\end{equation*}
	Similarly, with $a(t):=\|y(t)-y_d(t)\|_{L^2(0,1)}$ and the $\theta$-periodic function $\bar a(t):=\|\bar y(t)-y_d(t)\|_{L^2(0,1)}$, using $|a-\bar a|\leq\|y-\bar y\|_{L^2(0,1)}\leq g$,
	\begin{equation*}
		|A_y|\leq \int_0^T(a-\bar a)^2+2\int_0^Tg\,\bar a+2\|\bar a\|^2_{L^2(0,\theta)}
		\leq \frac{2M^2}{\mu}+2c_\mu\|\bar a\|_{L^2(0,\theta)}+2\|\bar a\|^2_{L^2(0,\theta)} ,
	\end{equation*}
	and, by Lemma~\ref{Energy}(i), $\|\bar a\|_{L^2(0,\theta)}\leq\sqrt\theta\,\|\bar y\|_{C([0,\theta];L^2(0,1))}+\sqrt\theta\,\normC{y_d}\leq\sqrt\theta\, C_2\|\bar u\|_{L^2(0,\theta;L^2(\omega))}+\sqrt\theta\,\normC{y_d}$.
	Combining the last three estimates and using $(p+q)^2\leq2p^2+2q^2$, we obtain \eqref{lemma_estimate_functional_conclusion} with $C_4$ as in \eqref{C4}.
\end{proof}

The second lemma bounds the value of $(OCP)_T$ from above in terms of the value of $(Per)_\theta$.
\begin{lemma}\label{upper_bound_value}
	Let $R_3\in\big(0,\frac{1}{2\sqrt{2\theta}+2\sqrt{4\theta+2}}\big)$ and let $y_d\in C([0,\infty);L^2 (0,1))$ be $\theta$-periodic with $ \normC{y_d}\leq R_3$. Then, for every $y_0\in L^2(0,1)$ and every $T>0$,
	\begin{equation}\label{value_function_ineq}
		\inf_{u\in L^2(0,T;L^2 (\omega))} J_{T}(u)\leq \frac{T}{\theta}\inf_{u\in L^2(0,\theta;L^2 (\omega))}J_\theta (u) + K_2,
	\end{equation}
	where
	\begin{equation}\label{K2}
		K_2:=2\theta R_3^2+\frac{\sqrt\theta R_3}{1-e^{-\theta/2}}\big(\|y_0\|_{L^2(0,1)}+C_2\sqrt\theta R_3\big)+\frac{1}{2}\big(\|y_0\|_{L^2(0,1)}+C_2\sqrt\theta R_3\big)^2
	\end{equation}
	is independent of $T$.
\end{lemma}
\begin{proof}
	Let $(y_{\theta,*},u_{\theta,*})$ be an optimal pair of $(Per)_\theta$, extended by $\theta$-periodicity, and let $\hat y$ be the solution of \eqref{state} with control $u_{\theta,*}|_{(0,T)}$ and initial datum $y_0$. The difference $\eta:=\hat y-y_{\theta,*}$ solves
	\begin{equation*}
		\partial_t\eta-\partial_{xx}\eta+\eta\,\partial_x\eta+\partial_x\left(y_{\theta,*}\eta\right)=0,\qquad \eta(t,0)=\eta(t,1)=0,\qquad \eta(0)=y_0-y_{\theta,*}(0).
	\end{equation*}
	Multiplying by $\eta$ and using $\int_0^1\eta^2\partial_x\eta=0$ and $-\int_0^1\partial_x(y_{\theta,*}\eta)\,\eta=\int_0^1 y_{\theta,*}\,\eta\,\partial_x\eta$, we get
	\begin{equation*}
		\frac{1}{2}\frac{\d}{\d t} \norm{\eta}^2_{L^2(0,1)}+\norm{\partial_x\eta}^2_{L^2(0,1)} = \int_0^1 y_{\theta,*}\,\eta\,\partial_x\eta .
	\end{equation*}
	By \eqref{lemma1.3.} and \eqref{uthetasmall}, $\|y_{\theta,*}(t)\|_{L^2(0,1)}\leq C_2\sqrt\theta\,\normC{y_d}\leq C_2\sqrt{\theta}R_3<\frac{1}{2\sqrt2}$ for all $t$; hence, by the coercivity estimate \eqref{ineq:corecive} (with $u=y_{\theta,*}(t)$, $\varepsilon=1$) and the Poincar\'e inequality, exactly as in \eqref{uniform_coercive_rho},
$\frac{\d}{\d t}\norm{\eta}^2_{L^2(0,1)}+\norm{\eta}^2_{L^2(0,1)}\leq0$, whence $\|\eta(t)\|^2_{L^2(0,1)}\leq e^{-t}\,\|\eta(0)\|^2_{L^2(0,1)}$. Now,
	\begin{equation*}
		\inf J_T\leq J_T(u_{\theta,*})=\frac{T}{\theta}J_\theta(u_{\theta,*})+\frac{1}{2} B_u+\frac{1}{2} B_y ,
	\end{equation*}
with 
$$
B_u:=\int_0^T\|u_{\theta,*}\|^2_{L^2(\omega)}-\frac T\theta\int_0^\theta\|u_{\theta,*}\|^2_{L^2(\omega)},
\quad
B_y:=\int_0^T\|\hat y-y_d\|^2_{L^2(0,1)}-\frac T\theta\int_0^\theta\|y_{\theta,*}-y_d\|^2_{L^2(0,1)} .
$$
By \eqref{periodleftover} and \eqref{uthetasmall}, $B_u\leq2\|u_{\theta,*}\|^2_{L^2(0,\theta;L^2(\omega))}\leq2\theta R_3^2$. Expanding $\|\hat y-y_d\|^2=\|y_{\theta,*}-y_d\|^2+2\langle\eta,y_{\theta,*}-y_d\rangle+\|\eta\|^2$ and using \eqref{periodleftover}, the exponential decay of $\eta$ and, as in \eqref{gphi} (with $\mu=\frac{1}{2}$ and only the term $e^{-t/2}$), the periodicity of $\bar a(t):=\|y_{\theta,*}(t)-y_d(t)\|_{L^2(0,1)}$,
	\begin{multline*}
		B_y\leq 2\int_0^\theta\bar a^2+2\|\eta(0)\|_{L^2(0,1)}\int_0^Te^{-t/2}\,\bar a(t)\d t+\int_0^T\|\eta\|^2_{L^2(0,1)}\\
		\leq 4J_\theta(u_{\theta,*})+\frac{2\|\eta(0)\|_{L^2(0,1)}\|\bar a\|_{L^2(0,\theta)}}{1-e^{-\theta/2}}+\|\eta(0)\|^2_{L^2(0,1)} .
	\end{multline*}
	Since $J_\theta(u_{\theta,*})=\inf J_\theta$, and since $2J_\theta(u_{\theta,*})\leq2J_\theta(0)\leq\theta R_3^2$, $\|\bar a\|_{L^2(0,\theta)}\leq\sqrt{2J_\theta(0)}\leq\sqrt\theta\,R_3$ and $\|\eta(0)\|_{L^2(0,1)}\leq\|y_0\|_{L^2(0,1)}+C_2\sqrt\theta R_3$, collecting the above estimates yields \eqref{value_function_ineq} with $K_2$ as in \eqref{K2}.
\end{proof}

\subsection{Global absorbing property and proof of Theorem~\ref{main1}}\label{proof_main1}
Let $ M_0 $ and $\mu$  be given by Proposition~\ref{local_turnpike}. 
For any fixed $M\in(0,M_0]$, let $\varepsilon_M$ be the threshold in \eqref{smallness_target_datum_varepsilon}. 
Denote by $y_{T,*}(\cdot;y_0,y_d)$ an optimal state of $(OCP)_T$. To define the absorbing time independently of the choice of an almost-everywhere representative, set
$$
E_M:=\big\{t\in(0,T) \,\mid\, y_{T,*}(t;y_0,y_d)\in H_0^1(0,1),\ \|y_{T,*}(t;y_0,y_d)\|_{H_0^1(0,1)} \leq\varepsilon_M/2\big\}
$$
and
\begin{equation}\label{t_s_definition}
t_M=t_M(y_0,y_d,T):=\inf\big\{s\in[0,T]\,\mid\, |E_M\cap(0,s)|>0\big\},
\end{equation}
with the convention $t_M=T$ if the set on the right-hand side is empty. Since $y_{T,*}\in W(0,T)$, one has $y_{T,*}(t)\in H_0^1(0,1)$ for a.e. $t\in(0,T)$. The definition implies
$\|y_{T,*}(t)\|_{H_0^1(0,1)}>\varepsilon_M/2$ for a.e. $t\in(0,t_M)$.
The next lemma shows that $t_M$ is bounded independently of $T$, and that beyond (essentially) the absorbing time the optimal triple satisfies the local turnpike estimate.
\begin{lemma}[Global absorbing property]\label{lemma_absorbing_time}
	Let $R>0$ and $M\in(0,M_0]$. There exist $\rho_{M}>0$ and $\tau_{M}:=\tau_M(\theta,\omega,R,M)>0$ such that the following holds:
	for every $y_0\in L^2(0,1)$ with $\|y_0\|_{L^2(0,1)}\leq R$, every $\theta$-periodic $y_d\in C([0,\infty);L^2(0,1))$ satisfying $\normC{y_d}\leq \rho_M$, every $T\geq \tau_M$ and every optimal triple $\left(y_{T,*},u_{T,*},\lambda_{T,*}\right)$ of $(OCP)_T$, there exists $\hat t_M\in(0,\tau_M]$ such that $\|y_{T,*}(\hat t_M)\|_{H_0^1(0,1)}\leq\varepsilon_M/2$ and
	\begin{multline}\label{lemma_opt_est_eq2}
		\|y_{T,*}(t) - y_{\theta,*}(t)\|_{L^2(0,1)}+\|\lambda_{T,*}(t)-\lambda_{\theta,*}(t)\|_{L^2(0,1)} + \|u_{T,*}(t) - u_{\theta,*}(t)\|_{L^2(\omega)} \\ \leq M\big(e^{-\mu (t-\hat t_M) }+e^{-\mu(T-t)}\big), \quad\text{a.e. } t \in[\hat t_M,T].
	\end{multline}
\end{lemma}
\begin{proof}
	Let $(y_{\theta,*},u_{\theta,*})$ be an optimal pair of $ (Per)_\theta $, extended by $ \theta $-periodicity over the whole real line. Throughout the proof we take
	\begin{equation}\label{rhoM}
		\rho_M:=\min\Big(\frac{1}{4\sqrt{2\theta}+4\sqrt{4\theta+2}}, \frac{\varepsilon_M}{2\sqrt2}\Big) , 
	\end{equation}
	so that, by \eqref{lemma_opt_est_eq3} below, $\frac1\theta\inf J_\theta\leq\frac{\varepsilon_M^2}{16}$. Indeed,
	\begin{equation}\label{lemma_opt_est_eq3}
		\frac{1}{2}\|u_{\theta,*}\|_{L^2(0,\theta;L^2(\omega))}^2 \leq \inf_{u\in L^2(0,\theta;L^2(\omega))}J_\theta= J_\theta(u_{\theta,*})\leq J_\theta(0)\leq \frac{\theta}{2} \normC{y_d}^2\leq\frac{\theta}{2}\rho_M^2\leq\frac{\theta\,\varepsilon_M^2}{16} .
	\end{equation}
	Moreover, since $\rho_M<\frac{1}{2\sqrt{2\theta}+2\sqrt{4\theta+2}}$, Lemma~\ref{upper_bound_value} and the optimality of $u_{T,*}$ give
	\begin{equation}\label{lemma_opt_est_eq8}
		J_{T}\left(u_{T,*}\right)\leq \frac{T}{\theta}\inf_{u\in L^2(0,\theta;L^2(\omega))}J_\theta+K_2 ,
	\end{equation}
	where $K_2=K_2(\theta,\rho_M,R)$ is given by \eqref{K2} (it is nondecreasing in $\|y_0\|_{L^2(0,1)}\leq R$).
	
	\smallskip
	
	\emph{Step 1: the absorbing set is nonempty for $T$ large.} Assume by contradiction that, for a.e.\ $t\in(0,T)$, one has $\|y_{T,*}(t)\|_{H_0^1(0,1)}>\varepsilon_M/2$. Integrating \eqref{diffineq} (with $u=u_{T,*}$) over $(0,T)$ gives $\int_0^T\|y_{T,*}\|^2_{H_0^1}\leq\|y_0\|^2_{L^2}+\int_0^T\|u_{T,*}\|^2_{L^2(\omega)}$, whence
	\begin{equation}\label{JTlower}
		J_T(u_{T,*}) \geq\frac{1}{2}\int_0^T \|y_{T,*}(t)\|_{H_0^1(0,1)}^2 \d t - \frac{R^2}{2} \geq \frac{T\varepsilon_M^2}{8} - \frac{R^2}{2} .
	\end{equation}
	Together with \eqref{lemma_opt_est_eq8} and $\frac1\theta\inf J_\theta\leq\frac{\varepsilon_M^2}{16}$, this yields $\frac{T\varepsilon_M^2}{16}\leq K_2+\frac{R^2}{2}$, which is impossible for $T>\frac{16K_2+8R^2}{\varepsilon_M^2}$. Hence the absorbing set has positive measure for such $T$, and by \eqref{t_s_definition} we may choose
$\hat t_M\in E_M\cap[t_M,t_M+1]$; in particular, $\|y_{T,*}(\hat t_M)\|_{H_0^1(0,1)}\leq \varepsilon_M/2$ and $\hat t_M\leq t_M+1$.
	
	\smallskip

	\emph{Step 2: turnpike on the tail interval.} By the Bellman optimality principle (see, e.g., \cite[Theorem~1.1 in Chapter 6]{LY}), the restriction $u_{T,*}|_{(\hat t_M,T)}$, shifted in time, is optimal for the problem $ (OCP)_{T-\hat t_M} $ with initial datum $y_{T,*}(\hat t_M)$ and tracking term $y_d(\cdot+\hat t_M)$. The shifted tracking term is $\theta$-periodic with $\normC{y_d(\cdot+\hat t_M)}=\normC{y_d}\leq\rho_M\leq\varepsilon_M/2$; using the definition of $\hat t_M$ this gives $\|y_{T,*}(\hat t_M)\|_{H_0^1(0,1)}+\normC{y_d(\cdot+\hat t_M)}\leq\varepsilon_M$, so that Proposition~\ref{local_turnpike} applies to the shifted problem. Its constants are unchanged by this phase shift, by the uniformity assertion in Proposition~\ref{local_turnpike}. Moreover, the cost of $(Per)_\theta$ is invariant under time shifts of the tracking term, and the shift $\big(y_{\theta,*}(\cdot+\hat t_M),u_{\theta,*}(\cdot+\hat t_M),\lambda_{\theta,*}(\cdot+\hat t_M)\big)$ of the (unique) optimal triple of $(Per)_\theta$ is an optimal triple for the shifted periodic problem; by uniqueness (Proposition~\ref{local_uniqueness}), it is \emph{the} optimal triple of the shifted periodic problem. Hence \eqref{turnpike}, applied to the shifted problems and written in the original time variable, gives precisely \eqref{lemma_opt_est_eq2}.
	
	\smallskip

	\emph{Step 3: the absorbing time is uniformly bounded.} By the definition \eqref{t_s_definition} of $t_M$, one has $\|y_{T,*}(t)\|_{H_0^1(0,1)}\geq \varepsilon_{M}/2$ for a.e.\ $t\in (0,t_M)$; hence, as in \eqref{JTlower},
	\begin{equation}\label{lemma_opt_est_eq11}
		\frac{1}{2}\int_0^{t_M}\|u_{T,*}\|_{L^2(\omega)}^2  \d t+\frac{1}{2}\int_{0}^{t_M}\|y_{T,*}-y_d\|_{L^2(0,1)}^2 \d t\geq \frac{t_M \varepsilon_{M}^2}{8}- \frac{R^2}{2}.
	\end{equation}
	On the other hand, the turnpike estimate \eqref{lemma_opt_est_eq2} allows us to apply Lemma~\ref{lemma_estimate_functional} on the tail interval, with horizon $T-\hat t_M$, control $u_{T,*}(\hat t_M+\cdot)$ and periodic control $\bar u:=u_{\theta,*}(\cdot+\hat t_M)$; recalling that $\|\bar u\|_{L^2(0,\theta;L^2(\omega))}=\|u_{\theta,*}\|_{L^2(0,\theta;L^2(\omega))}$ and $J_\theta^{shift}(\bar u)=\inf J_\theta$, and using \eqref{lemma_opt_est_eq3} and $\normC{y_d}\leq\rho_M\leq\frac{\varepsilon_M}{2}$, we get
	\begin{align}\label{lemma_opt_est_eq20}
		\frac{1}{2}\int_{\hat t_M}^{T}\|u_{T,*}\|_{L^2(\omega)}^2  \d t+\frac{1}{2}\int_{\hat t_M}^{T}\|y_{T,*}-y_d\|_{L^2(0,1)}^2 \d t
		&\geq\frac{T-\hat t_M}{\theta}\inf_{u\in L^2(0,\theta;L^2(\omega))}J_\theta- C_4\,\Theta_M ,
	\end{align}
	where $\Theta_M:=1+(1+\sqrt\theta)\frac{\varepsilon_M}{2}+\big(\frac38+\frac{\theta}{4}\big)\varepsilon_M^2$ and $C_4=C_4(M,\mu,\theta)$ is given by \eqref{C4}. Since the running cost is nonnegative, splitting $J_T(u_{T,*})$ on $(0,t_M)$, $(t_M,\hat t_M)$ and $(\hat t_M,T)$ and using \eqref{lemma_opt_est_eq11}, \eqref{lemma_opt_est_eq20}, $T-\hat t_M\geq T-t_M-1$ and $\frac{1}{\theta}\inf J_\theta\leq\frac{\varepsilon_M^2}{16}$,
	\begin{equation}\label{lemma_opt_est_eq24}
		J_{T}\left(u_{T,*}\right)\geq  \frac{t_M \varepsilon_{M}^2}{8}- \frac{R^2}{2}+\frac{T-t_M}{\theta}\inf_{u}J_\theta-\frac{\varepsilon_M^2}{16}-C_4\,\Theta_M .
	\end{equation}
	Comparing \eqref{lemma_opt_est_eq24} with \eqref{lemma_opt_est_eq8} and using once more $\frac{1}{\theta}\inf J_\theta\leq\frac{\varepsilon_M^2}{16}$, we obtain
	\begin{equation*}
		\frac{t_M\varepsilon_M^2}{16} \leq t_M\Big(\frac{\varepsilon_M^2}{8}-\frac{1}{\theta}\inf_u J_\theta\Big) \leq K_2+\frac{R^2}{2}+\frac{\varepsilon_M^2}{16}+C_4\,\Theta_M ,
	\end{equation*}
	whence $t_M\leq \frac{16 K_2+8R^2+16\,C_4\Theta_M}{\varepsilon_M^2}+1$. Setting
	\begin{equation}\label{tauM}
		\tau_{M}:= \frac{16 K_2+8R^2+16\,C_4\Theta_M}{\varepsilon_M^2}+2 ,
	\end{equation}
	we conclude that $t_M\leq\tau_M-1$, and hence $\hat t_M\leq t_M+1\leq\tau_M$, for every $T\geq\tau_M$ (note that $\tau_M>\frac{16K_2+8R^2}{\varepsilon_M^2}$, so that Step~1 applies). 
\end{proof}

We are now in a position to prove Theorem~\ref{main1}.
\begin{proof}[Proof of Theorem~\ref{main1}]
	Fix $M\in(0,M_0]$ (for instance $M=M_0$), let $\varepsilon_M$ be as in Proposition~\ref{local_turnpike}, and let $R:=1+\|y_0\|_{L^2(0,1)}$. Take $\rho:=\rho_M$ as in \eqref{rhoM} and $\tau:=\tau_M$ as in \eqref{tauM}; both are independent of $T$, and $\rho$ is independent of $y_0$. Since $\rho_M\leq\varepsilon_M\leq\varepsilon$ (see \eqref{choice_epsM} in Appendix~\ref{app:C}), the condition $\normC{y_d}\leq\rho$ ensures, by Proposition~\ref{local_uniqueness}, that $(Per)_\theta$ has a unique optimal pair $(y_{\theta,*},u_{\theta,*})$.
	Let now $T\geq\tau$ and let $(y_{T,*},u_{T,*},\lambda_{T,*})$ be any optimal triple of $(OCP)_T$. By Lemma~\ref{lemma_absorbing_time}, there exists $\hat t_M\in(0,\tau_M]$ such that \eqref{lemma_opt_est_eq2} holds. Hence, for a.e.\ $t\in[\tau,T]\subset[\hat t_M,T]$,
	\begin{equation*}
		\|y_{T,*}(t) - y_{\theta,*}(t)\|_{L^2(0,1)} + \|u_{T,*}(t) - u_{\theta,*}(t)\|_{L^2(\omega)}
		\leq Me^{\mu\hat t_M}e^{-\mu t}+Me^{-\mu(T-t)}
		\leq K\big(e^{-\mu t }+e^{-\mu(T-t)}\big),
	\end{equation*}
	with $K:=e^{\mu \tau_{M}} M$, which is independent of $T$. This is \eqref{turnpike1}, and the proof is complete.
\end{proof}


\begin{appendices}

\section{\texorpdfstring{Uniform slice bounds: proof of \eqref{uniform_bound_optima1}}{Uniform slice bounds}}\label{app:A}

Throughout this appendix we take $\bar T=4$ in \eqref{def:slices}. Thus $N_T=\lfloor T/4\rfloor$, the full slices are $I_i=(4(i-1),4i)$, $1\leq i\leq N_T$, and $I_{N_T+1}=(4N_T,T)$ is the possibly empty leftover slice.

For an interval $J\subset\mathbb R$, set
\begin{equation}\label{defY}
	\mathcal Y(J):=L^2(J;H^2(0,1)\cap H_0^1(0,1))\cap H^1(J;L^2(0,1)).
\end{equation}
In particular, $\mathcal Y(J)$ embeds continuously into $C(\overline J;H_0^1(0,1))$.

\begin{lemma}\label{exact_control}
Let $\varrho>0$. There exist constants $\varepsilon_c>0$ and $C_c>0$, depending only on $\varrho$ and $\omega$, such that, for any $s\geq0$ and any $(\hat y,\hat u)\in\mathcal Y((s,s+1))\times L^2(s,s+1;L^2(\omega))$ satisfying
	\begin{equation*}
		\partial_t\hat y-\partial_{xx}\hat y+\hat y\,\partial_x\hat y=\chi_\omega\hat u,\qquad
		\hat y|_{\{0,1\}}=0,\qquad
		\|\hat y\|_{\mathcal Y((s,s+1))}\leq\varrho ,
	\end{equation*}
	for any $z_s\in H_0^1(0,1)$ such that $\|z_s-\hat y(s)\|_{H_0^1(0,1)}\leq\varepsilon_c$, there exists $v\in L^2(s,s+1;L^2(\omega))$ such that the solution $z$ of
	$$
		\partial_tz-\partial_{xx}z+z\,\partial_xz=\chi_\omega(\hat u+v),\qquad
		z|_{\{0,1\}}=0,\qquad z(s)=z_s,
	$$
	satisfies $z(s+1)=\hat y(s+1)$ and 
	\begin{equation}\label{exact_control_cost}
		\|v\|_{L^2(s,s+1;L^2(\omega))}
		\leq C_c\|z_s-\hat y(s)\|_{H_0^1(0,1)} .
	\end{equation}
	The constants are unchanged under time translation.
\end{lemma}

\begin{proof}
	We give a proof for completeness.
	Choose a nonempty closed interval $\omega'\Subset\omega$; controls supported in $\omega'$ are identified with elements of $L^2(\omega)$. By a time translation it suffices to consider $s=0$. We use the argument of \cite[equations~(1.2)--(1.4) and proof of the Main Theorem]{SA}, based on \cite[Theorem~4.3]{FI95}. Write $w=z-\hat y$ and $w_0=z_s-\hat y(0)$. We have to solve
	\begin{equation}\label{wbridge}
		\partial_tw-\partial_{xx}w+\partial_x\left(\left(\hat y+\frac{w}{2}\right)w\right)
		=\chi_\omega v,\qquad w(0)=w_0,\qquad w(1)=0 .
	\end{equation}
	For every $a_0\in\mathcal Y((0,1))$, the cited linear controllability result furnishes a bounded linear operator $	\mathcal C_{a_0}:H_0^1(0,1)\longrightarrow L^2(0,1;L^2(\omega))$ such that the solution $\mathcal S(a_0,w_0)\in\mathcal Y((0,1))$ of $\partial_tw-\partial_{xx}w+\partial_x(a_0w) = \chi_\omega\mathcal C_{a_0}w_0$, $w(0)=w_0$, vanishes at $t=1$. Moreover, for every $R>0$ there is $C_R>0$ such that
	\begin{equation}\label{linear_FI_bound}
		\|\mathcal S(a_0,w_0)\|_{\mathcal Y((0,1))} + \|\mathcal C_{a_0}w_0\|_{L^2(0,1;L^2(\omega))} \leq C_R\|w_0\|_{H_0^1}
	\end{equation}
	whenever $\|a_0\|_{\mathcal Y((0,1))}\leq R$. On bounded subsets, the map $a_0\mapsto\mathcal S(a_0,w_0)$ is continuous and has relatively compact image in $\mathcal Y((0,1))$ (see the proof cited above). For $q$ in the closed unit ball of $\mathcal Y((0,1))$, define $\Phi(q):=\mathcal S(\hat y+\frac q2,w_0)$. The coefficients $\hat y+q/2$ remain in the ball of radius $\varrho+1$ in $\mathcal Y((0,1))$. Hence $\Phi$ is continuous and compact and, by \eqref{linear_FI_bound},
	$\|\Phi(q)\|_{\mathcal Y((0,1))}\leq C_{\varrho+1}\|w_0\|_{H_0^1}$.
	If $\|w_0\|_{H_0^1}\leq C_{\varrho+1}^{-1}$, Schauder's theorem gives a fixed point $w=\Phi(w)$ in that ball. It solves \eqref{wbridge} with
	$v=\mathcal C_{\hat y+w/2}w_0$, and \eqref{linear_FI_bound} yields \eqref{exact_control_cost}. We may therefore take $\varepsilon_c=C_{\varrho+1}^{-1}$ and $C_c=C_{\varrho+1}$.
\end{proof}

We next give the two regularization estimates used to place the reference trajectory in the space required by Lemma~\ref{exact_control} and to regularize the freely evolving competitor before the bridge.

\begin{lemma}[Strong regularity and free smoothing]\label{lem:local_smoothing}
	Let $s\geq0$, $z_s\in H_0^1(0,1)$ and $g\in L^2(s,s+1;L^2(0,1))$. Let $z$ solve
	\begin{equation}\label{forced_burgers_reg}
		\partial_tz-\partial_{xx}z+z\partial_xz=g,\qquad
		z|_{\{0,1\}}=0,\qquad z(s)=z_s .
	\end{equation}
	Set $E_s:=\|\partial_xz_s\|_{L^2}^2+2\|g\|_{L^2(s,s+1;L^2)}^2$.
	If $E_s\leq1/8$, then $z\in\mathcal Y((s,s+1))$ and
	\begin{equation}\label{strong_reg}
		\|z\|_{\mathcal Y((s,s+1))}
		\leq C_{\mathrm{reg}}\big(\|z_s\|_{H_0^1}+\|g\|_{L^2(s,s+1;L^2)}\big),
	\end{equation}
	for a universal constant $C_{\mathrm{reg}}>0$.
	
	Moreover, let $\xi\in L^2(0,1)$ with $\|\xi\|_{L^2}\leq\pi/2$, and let $z$ be the uncontrolled solution starting from $\xi$ at time $r$. Then
	\begin{equation}\label{free_smoothing}
		\|z(t)\|_{H_0^1}\leq\|\xi\|_{L^2}\quad\forall t\geq r+\frac12,
		\qquad
		\int_r^\infty\|z(t)\|_{L^2}^2 \d t\leq\frac{1}{2\pi^2}\|\xi\|_{L^2}^2 .
	\end{equation}
\end{lemma}

\begin{proof}
	The calculations can be justified by a Galerkin approximation. Set $q(t):=\|\partial_xz(t)\|_{L^2}^2$. Testing \eqref{forced_burgers_reg} by $-\partial_{xx}z$ and using $\|z\|_{L^\infty}\leq\|\partial_xz\|_{L^2}$ gives
	\begin{equation}\label{qstrong}
		q'(t)+\|\partial_{xx}z(t)\|_{L^2}^2 \leq 2q(t)^2+2\|g(t)\|_{L^2}^2 .
	\end{equation}
	A continuity argument yields
	\begin{equation}\label{qbootstrap}
		\sup_{t\in[s,s+1]}q(t)\leq2E_s,\qquad
		\int_s^{s+1}\|\partial_{xx}z(t)\|_{L^2}^2 \d t\leq2E_s .
	\end{equation}
	Indeed, as long as $q\leq2E_s$, integration of \eqref{qstrong} gives $q(t)\leq E_s+8E_s^2\leq2E_s$; the integral estimate follows in the same way. Finally,
	$$
		\partial_tz=\partial_{xx}z-z\partial_xz+g,\qquad
		\|z\partial_xz\|_{L^2}^2\leq\|\partial_xz\|_{L^2}^4=q(t)^2.
	$$
	Thus \eqref{qbootstrap}, Poincar\'e's inequality and the equivalence of the $H^2\cap H_0^1$ norm with $\|\partial_{xx}z\|_{L^2}$ prove \eqref{strong_reg}.
	
	For the free solution, the exact energy identity
	\begin{equation}\label{free_exact_energy}
		\frac12\frac{\d}{\d t}\|z(t)\|_{L^2}^2+\|\partial_xz(t)\|_{L^2}^2=0
	\end{equation}
	gives the integral estimate in \eqref{free_smoothing}. It also provides a time $\sigma\in(r,r+1/2)$ such that
	$\|\partial_xz(\sigma)\|_{L^2}^2\leq\|\xi\|_{L^2}^2$.
	For $t>\sigma$, testing by $-\partial_{xx}z$ and using
	$\|\partial_xz\|_{L^2}\leq\pi^{-1}\|\partial_{xx}z\|_{L^2}$ gives
	$\frac12q'(t)+\|\partial_{xx}z\|_{L^2}^2 \leq \frac{\sqrt{q(t)}}{\pi}\|\partial_{xx}z\|_{L^2}^2$.
	Since $\sqrt{q(\sigma)}\leq\|\xi\|_{L^2}\leq\pi/2$, a continuity argument shows that $q$ is nonincreasing on $(\sigma,\infty)$, proving the first estimate in \eqref{free_smoothing}.
\end{proof}

Set $M_b:=48$ and $A_b:=1+M_b$ (these constants, as well as the constants introduced in the proof of Lemma~\ref{prop4.1} below, are local to this appendix).
Let $\varepsilon_c:=\varepsilon_c(1)$ and $C_c:=C_c(1)$ be given by Lemma~\ref{exact_control}, set $C_m:=\sqrt{A_b}+\sqrt{2M_b}$, and fix $D_*\in(0,1]$ so small that
\begin{equation}\label{Dstar}
		A_bD_*\leq\frac{\pi^2}{4},\quad 4M_bD_*\leq\frac18,\quad
		C_{\mathrm{reg}}\big(\sqrt{2M_b}+\sqrt{M_b}\big)\sqrt{D_*}\leq1,\quad
		C_m\sqrt{D_*}\leq\varepsilon_c.
\end{equation}
We take $r_{\mathrm{ub}}:=\sqrt{D_*}$.

\begin{lemma}\label{prop4.1}
	There exists $K_{\mathrm{ub}}>0$, depending only on $\omega$ and independent of $T$, such that, for all $T>0$, $y_0\in L^2(0,1)$ and $\theta$-periodic tracking terms $y_d\in C([0,\infty);L^2(0,1))$ satisfying $\|y_0\|_{L^2(0,1)}+\normC{y_d}\leq r_{\mathrm{ub}}$, every optimal pair $(y_{T,*},u_{T,*})$ of $(OCP)_T$ satisfies
	$$
		\max_{1\leq i\leq N_T+1}\int_{I_i} \big(\|u_{T,*}(t)\|_{L^2(\omega)}^2+\|y_{T,*}(t)\|_{L^2(0,1)}^2\big)\d t
		\leq K_{\mathrm{ub}}\big(\|y_0\|_{L^2(0,1)}^2+\normC{y_d}^2\big).
	$$
\end{lemma}

\begin{proof}
	Let $D:=\|y_0\|_{L^2(0,1)}^2+\normC{y_d}^2\leq D_*$, $f(t):=\|u_{T,*}(t)\|_{L^2(\omega)}^2+\|y_{T,*}(t)\|_{L^2(0,1)}^2$ and $\delta:=\normC{y_d}$.
	If $D=0$, the null pair has zero cost and the conclusion is immediate. We henceforth assume $D>0$.
	
	We use two elementary observations. First, if an interval $J$ satisfies $\int_Jf\leq M_bD$, then every unit subinterval of $J$ contains a Lebesgue point $t$ such that $f(t)\leq M_bD$. Second, for $0\leq t_1<t_2\leq T$, the restriction $u_{T,*}|_{(t_1,t_2)}$ minimizes
	$$
		J_{t_1,t_2}(u):=\frac12\int_{t_1}^{t_2}\big(\|u(t)\|_{L^2(\omega)}^2+\|y(t)-y_d(t)\|_{L^2(0,1)}^2\big)\d t
	$$
	among controls whose state starts from $y_{T,*}(t_1)$ and reaches $y_{T,*}(t_2)$ at time $t_2$. Otherwise, a cheaper control could be spliced into $u_{T,*}$. When $t_2=T$, the terminal constraint can be dropped: $u_{T,*}|_{(t_1,T)}$ minimizes $J_{t_1,T}$ among \emph{all} controls whose state starts from $y_{T,*}(t_1)$, for the same splicing reason; this free-endpoint version is the one used in Case~II and for the leftover slice below. We shall also use
	\begin{equation}\label{conversion_appendix}
		\int_{t_1}^{t_2}f(t) \d t
		\leq4J_{t_1,t_2}(u_{T,*}|_{(t_1,t_2)})+2(t_2-t_1)\delta^2 .
	\end{equation}
	
	\emph{Short horizons.}
	Let $z$ be the uncontrolled solution starting from $y_0$. By optimality, \eqref{free_exact_energy} and $\frac12\|z-y_d\|^2\leq\|z\|^2+\|y_d\|^2$, we have $J_T(u_{T,*})\leq J_T(0)\leq\frac{1}{2\pi^2}\|y_0\|_{L^2}^2+T\delta^2$. Hence, if $T\leq8$,
	\begin{equation*}
		\int_0^Tf\leq4J_T(u_{T,*})+2T\delta^2 \leq \Big(\frac{2}{\pi^2}+48\Big)D.
	\end{equation*}
	We assume from now on that $T>8$.
	
	\smallskip
	\emph{Good and bad slices.}
	A full slice $I_i$, $1\leq i\leq N_T$, is called good if $\int_{I_i}f\leq M_bD$, and bad otherwise. The union of the bad slices is a finite disjoint union of maximal clusters $[a_j,b_j]$, each made of consecutive full slices. Fix one such cluster and write it as $[a,b]$. If $a=0$, set $t_1=0$. If $a>0$, the preceding slice $[a-4,a]$ is good, and we choose $t_1\in(a-1,a)$ such that $f(t_1)\leq M_bD$.
	In both cases,
	\begin{equation}\label{yt1bound}
		\|y_{T,*}(t_1)\|_{L^2}^2\leq A_bD .
	\end{equation}
	
	\emph{Case I: the cluster is followed by a full good slice.}
	Assume $b<4N_T$. Then $[b,b+4]$ is good. Choose $t_c\in(b,b+1)$ such that $f(t_c)\leq M_bD$. The energy estimate \eqref{diffineq} gives the two bounds
	\begin{equation*}
		\sup_{t\in[t_c,b+3]}\|y_{T,*}(t)\|_{L^2}^2\leq2M_bD,\qquad
		\int_{t_c}^{b+3}\|\partial_xy_{T,*}(t)\|_{L^2}^2 \d t\leq2M_bD .
	\end{equation*}
	We may therefore choose a Lebesgue time $s\in(b+1,b+2)$ such that
	\begin{equation}\label{sample_H1}
		\|y_{T,*}(s)\|_{H_0^1}^2\leq2M_bD .
	\end{equation}
	Moreover,
	\begin{equation}\label{controlwindow}
		\|u_{T,*}\|_{L^2(s,s+1;L^2(\omega))}^2\leq M_bD .
	\end{equation}
	Lemma~\ref{lem:local_smoothing} and \eqref{Dstar} imply $\|y_{T,*}\|_{\mathcal Y((s,s+1))}
		\leq C_{\mathrm{reg}}\big(\sqrt{2M_b}+\sqrt{M_b}\big)\sqrt D\leq1$.
	
	Let $z$ be the uncontrolled solution starting from $y_{T,*}(t_1)$ at time $t_1$. Since the cluster contains at least one full slice and $s>b+1$, we have $s-t_1>1/2$. By \eqref{yt1bound}, \eqref{Dstar} and \eqref{free_smoothing}, we have $\|z(s)\|_{H_0^1}\leq\|y_{T,*}(t_1)\|_{L^2}\leq\sqrt{A_bD}$.
	Together with \eqref{sample_H1}, this yields $\|z(s)-y_{T,*}(s)\|_{H_0^1} \leq C_m\sqrt D\leq\varepsilon_c$.
	Lemma~\ref{exact_control}, applied with $\hat y=y_{T,*}|_{(s,s+1)}$ and $\hat u=u_{T,*}|_{(s,s+1)}$, furnishes $v\in L^2(s,s+1;L^2(\omega))$ such that the solution starting from $z(s)$ under the control $u_{T,*}+v$ reaches $y_{T,*}(s+1)$ and
	\begin{equation}\label{vbridge}
		\|v\|_{L^2(s,s+1;L^2(\omega))}\leq C_cC_m\sqrt D .
	\end{equation}
	Set $t_2:=s+1$ and define on $(t_1,t_2)$ the competitor $\tilde u(t):=0$ for $t\in(t_1,s)$ and $\tilde u(t):=u_{T,*}(t)+v(t)$ for $t\in(s,t_2)$. Its state $\tilde y$ reaches $y_{T,*}(t_2)$ at time $t_2$.
	
	Set $C_v:=C_cC_m$, $C_u:=2M_b+2C_v^2$, $C_y:=\frac{A_b}{2\pi^2}+A_b+C_u$, $C_0:=\frac{C_u}{2}+C_y$, $C_1:=4C_0$.
	By \eqref{controlwindow}, \eqref{vbridge}, the free decay on $(t_1,s)$ and the basic energy estimate on $(s,t_2)$, we have $\|\tilde u\|_{L^2(t_1,t_2;L^2(\omega))}^2\leq C_uD$ and $\int_{t_1}^{t_2}\|\tilde y(t)\|_{L^2}^2\d t\leq C_yD$.
	Consequently $J_{t_1,t_2}(\tilde u)\leq C_0D+(t_2-t_1)\delta^2$, and the exchange property and \eqref{conversion_appendix} give
	\begin{equation}\label{caseIupper}
		\int_{t_1}^{t_2}f(t) \d t\leq C_1D+6(t_2-t_1)\delta^2 .
	\end{equation}
	
	The cluster contains $(b-a)/4$ bad slices, while $t_1\in(a-1,a)$ (or $t_1=a=0$) and $t_2\in(b+2,b+3)$; hence $b-a\geq(t_2-t_1)-4$. It follows that
	\begin{equation}\label{caseIlower}
		\int_{t_1}^{t_2}f(t)\d t\geq\int_a^bf(t)\d t 	>\frac{b-a}{4}M_bD \geq 12 ((t_2-t_1)-4)D .
	\end{equation}
	Comparing \eqref{caseIupper} and \eqref{caseIlower} gives $t_2-t_1\leq\ell_I:=\frac{C_1+48}{6}$ and $\int_a^bf(t) \d t\leq K_ID$ with $K_I:=C_1+6\ell_I$.
	
	\smallskip
	\emph{Case II: the cluster reaches the final full slice.}
	Assume $b=4N_T$. Switch off the control after $t_1$, and denote the resulting free state by $\tilde y$. The free-endpoint version of the exchange property and \eqref{yt1bound} give $J_{t_1,T}(u_{T,*}|_{(t_1,T)}) \leq\frac{A_b}{2\pi^2}D+(T-t_1)\delta^2$.
	Thus, with $C_{II}:=2A_b/\pi^2$, the conversion estimate gives
	\begin{equation}\label{caseIIupper}
		\int_{t_1}^Tf(t) \d t\leq C_{II} D + 6(T-t_1)\delta^2.
	\end{equation}
	Since $0\leq T-b<4$ and $0\leq a-t_1<1$, one has $b-a\geq(T-t_1)-5$. Therefore $\int_{t_1}^Tf(t) \d t\geq\int_a^bf(t) \d t > 12 ((T-t_1)-5)D$.
	Comparison with \eqref{caseIIupper} yields $T-t_1\leq\ell_{II}:=\frac{C_{II}+60}{6}$ and $\int_a^Tf(t) \d t\leq K_{II}D,\quad K_{II}:=C_{II}+6\ell_{II}$.
	This also covers the leftover slice.
	
	\smallskip
	\emph{The leftover slice after a final good slice.}
	If $I_{N_T}$ is good and $I_{N_T+1}\neq\emptyset$, choose $t_1\in(4N_T-1,4N_T)$ with $f(t_1)\leq M_bD$. Switching off the control after $t_1$ and using $T-t_1<5$ gives $\int_{I_{N_T+1}}f(t) \d t \leq \big(\frac{2M_b}{\pi^2}+30\big)D=:K_{III}D$.

	The conclusion follows with $K_{\mathrm{ub}}:=\max\big(\frac{2}{\pi^2}+48, M_b, K_I, K_{II}, K_{III}\big)$.
\end{proof}

\section{Proof of Lemma~\ref{lem:convexity}}\label{app:B}
Throughout this appendix we set $\bar T=1$ in \eqref{def:slices}, and $K_{\mathrm{sl}}=K_{\mathrm{sl}}(1)$, $K_1=K_1(1)$ in \eqref{def:K} and Lemma~\ref{regular_lambda}.

\medskip
\noindent\emph{Proof of (i).} Define
\begin{equation}\label{r0def}
	r_0:=\min\Big(\frac{1}{8\sqrt2 K_{\mathrm{sl}}}, \frac{1}{108\,K_1}, 1\Big) ,
\end{equation}
let $r\in(0,r_0]$, and let $y_0,y_d$ be such that $\|y_0\|_{L^2(0,1)}\leq r_0$ and $\normC{y_d}\leq r_0$. Let $u\in B_r$, let $y_T$ be the corresponding state and $\lambda_T$ the corresponding adjoint state (without the feedback relation), and let $(z_T,v_T)$ solve \eqref{linearizedT}.

\smallskip

\emph{Step 1: state and adjoint bounds.} Since $\bar T=1$, $\sup_{i}\|u\|_{L^2(I_i;L^2(\omega))}\leq\|u\|_{L^\infty(0,T;L^2(\omega))}\leq r$, so Lemma~\ref{lem_y_upper_bound_by_sub_inter} gives
\begin{equation}\label{2.27}
	\|y_T\|_{C([0,T];L^2(0,1))}\leq K_{\mathrm{sl}}\left(\|y_0\|_{L^2(0,1)}+r\right)\leq 2K_{\mathrm{sl}}r_0\leq\frac{1}{4\sqrt2}<\frac{1}{2\sqrt2} .
\end{equation}
The adjoint equation in \eqref{maxim} only involves the state $y_T$; hence the proof of Lemma~\ref{regular_lambda} applies verbatim, with \eqref{y^T_upper_bounded_sub_inte} replaced by \eqref{2.27} (note that $4\sqrt2 K_{\mathrm{sl}}r_0\leq\frac{1}{2}<1$), and yields
\begin{equation}\label{bound_adjoint_estimate}
	\|\lambda_T\|_{L^\infty(0,T;L^2(0,1))}
	\leq K_1\left(\|y_0\|_{L^2(0,1)}+\normC{y_d}+r\right)\leq 3K_1r_0 .
\end{equation}

\smallskip

\emph{Step 2: bound for the linearized state.} Multiplying the equation \eqref{linearizedT} by $z_T$ and using $-\langle\partial_x(y_Tz_T),z_T\rangle=\int_0^1y_Tz_T\partial_xz_T$, the coercivity estimate \eqref{ineq:corecive} with $u=y_T(t)$, $\varepsilon=1$, the bound \eqref{2.27} and the Poincar\'e inequality, we get
\begin{equation}\label{zTenergy}
	\frac{1}{2}\frac{\d}{\d t}\norm{z_T}^2_{L^2(0,1)}+\frac14\norm{\partial_xz_T}^2_{L^2(0,1)}+3\norm{z_T}^2_{L^2(0,1)}\leq\int_\omega v_Tz_T\leq\frac{1}{2}\norm{v_T}^2_{L^2(\omega)}+\frac{1}{2}\norm{z_T}^2_{L^2(0,1)}
\end{equation}
(indeed, $1-\frac{3\sqrt2}{4}\|y_T\|_{L^2}\geq\frac58$, of which we keep $\frac14\|\partial_xz_T\|^2$ and convert the rest by Poincar\'e: $\frac38\pi^2-\frac{\sqrt2}{4}\|y_T\|_{L^2}\geq\frac{3\pi^2}{8}-\frac18\geq3$). Integrating \eqref{zTenergy} with $z_T(0)=0$ gives
\begin{equation}\label{zT}
	\sup_{t\in[0,T]}\|z_T(t)\|_{L^2(0,1)}
	+\|z_T\|_{L^2(0,T;H_0^1(0,1))}
	\leq 3\,\|v_T\|_{L^2(0,T;L^2(\omega))} .
\end{equation}

\emph{Step 3: the cross term.} By an integration by parts, the Agmon inequality and the Poincar\'e inequality, $\norm{z_T}^2_{L^\infty(0,1)}\leq2\norm{z_T}_{L^2}\norm{\partial_xz_T}_{L^2}\leq\frac2\pi\norm{\partial_xz_T}^2_{L^2}$, so that
\begin{equation}\label{eq_cross_T}
	\Big|\int_{0}^{T}\langle z_T\,\partial_x \lambda_T, z_T\rangle\d t\Big|=2\Big|\int_0^T\int_0^1\lambda_Tz_T\partial_xz_T\Big|
	\leq 2\sqrt{\tfrac2\pi}\,\|\lambda_T\|_{L^\infty(0,T;L^2 (0,1))} \|z_T\|_{L^2 (0,T;H_0^1(0,1))}^2 .
\end{equation}
Combining \eqref{eq_J_T_sec_der}, \eqref{bound_adjoint_estimate}, \eqref{zT} and \eqref{eq_cross_T},
\begin{equation*}
	J_T^{\prime\prime}(u)[v_T ,v_T ]\geq\big(1-54K_1r_0\big)\int_{0}^{T}\|v_T\|_{L^2(\omega)}^2\d t ,
\end{equation*}
and \eqref{convexT} follows from \eqref{r0def}. Since $B_r$ is convex and $J_T$ is twice Gateaux differentiable along segments of $B_r$ (see \cite[Theorem 5.2]{CT}), the estimate \eqref{convexT} implies the strict convexity of $J_T$ on $B_r$.

\medskip
\noindent\emph{Proof of (ii).} Let $\theta>0$. We first fix a preliminary radius and the corresponding constant of Lemma~\ref{Energy}(ii),
\begin{equation*}
	\bar r_\theta:=\min\big(\tfrac{1}{2\sqrt2\, C_2}, 1\big),
	\qquad
	C_3^0:=C_3(\theta,\bar r_\theta),
	\qquad
	K_z:=\big(\tfrac{1}{1-e^{-5\theta}}+1\big)^{1/2}+\sqrt2 ,
\end{equation*}
which is legitimate since $\sqrt2 C_2\bar r_\theta\leq\frac12$. We then define
\begin{equation}\label{rthetadef}
	r_\theta:=\min\Big(\bar r_\theta, \tfrac{1}{2K_z^2\,C_3^0\sqrt\theta\,(1+C_2)}\Big)\ \in(0,1] .
\end{equation}
Since $r_\theta\leq\bar r_\theta$, we have $\sqrt2 C_2r_\theta\leq\frac12$, so that Lemma~\ref{Energy}(ii) applies with $R_2=r_\theta$; moreover, inspecting the proof of Lemma~\ref{Energy}(ii), the constant $C_3(\theta,R_2)$ is nondecreasing in $R_2$ (the quantities $\kappa$ and $C_{R_2}$ are nonincreasing in $R_2$, hence $1/\kappa$ and $1/C_{R_2}$ are nondecreasing), so that $C_3(\theta,r_\theta)\leq C_3^0$. Assume $\normC{y_d}\leq r_\theta$ and let $u\in B^{per}_{r_\theta}$, with periodic state $y_\theta$, periodic adjoint $\lambda_\theta$ (Lemma~\ref{Energy}(ii)) and periodic linearized state $z_\theta$ solving \eqref{linearizedP} for a direction $v_\theta$ (the existence and uniqueness of the periodic solution $z_\theta$ follows, as for $\lambda_\theta$, from the contraction property of the period map implied by \eqref{zPenergy} below).

By Lemma~\ref{Energy}, $\|y_\theta\|_{C([0,\theta];L^2(0,1))}\leq C_2r_\theta\leq\frac{1}{2\sqrt2}$ and, using $C_3(\theta,r_\theta)\leq C_3^0$,
\begin{equation}\label{lambdaPbound}
	\|\lambda_\theta\|_{L^2(0,\theta;H_0^1(0,1))}\leq C_3^0\|y_\theta-y_d\|_{L^2(0,\theta;L^2(0,1))}\leq C_3^0\sqrt\theta\left(C_2r_\theta+\normC{y_d}\right)\leq C_3^0\sqrt\theta\,(1+C_2)\,r_\theta .
\end{equation}
Arguing exactly as in Step 2 above (the computation only uses $\|y_\theta\|_{C L^2}\leq\frac{1}{2\sqrt2}$),
\begin{equation}\label{zPenergy}
	\frac{\d}{\d t}\norm{z_\theta}^2_{L^2(0,1)}+\frac{1}{2}\norm{\partial_xz_\theta}^2_{L^2(0,1)}+5\norm{z_\theta}^2_{L^2(0,1)}\leq\norm{v_\theta}^2_{L^2(\omega)} .
\end{equation}
Integrating \eqref{zPenergy} over one period and using $z_\theta(0)=z_\theta(\theta)$ gives
$$\|z_\theta\|_{L^2(0,\theta;H_0^1(0,1))}\leq\sqrt2\,\|v_\theta\|_{L^2(0,\theta;L^2(\omega))}.$$
Moreover, Gr\"onwall's inequality applied to \eqref{zPenergy} over one period, combined with periodicity, gives $(1-e^{-5\theta})\|z_\theta(\theta)\|^2_{L^2}\leq\|v_\theta\|^2_{L^2(0,\theta;L^2(\omega))}$, whence
$$\sup_{[0,\theta]}\|z_\theta\|^2_{L^2}\leq\Big(\tfrac{1}{1-e^{-5\theta}}+1\Big)\|v_\theta\|^2_{L^2(0,\theta;L^2(\omega))}.$$
Altogether,
\begin{equation}\label{zP}
	\sup_{t\in[0,\theta]}\|z_\theta(t)\|_{L^2(0,1)}+\|z_\theta\|_{L^2(0,\theta;H_0^1(0,1))}\leq K_z\,\|v_\theta\|_{L^2(0,\theta;L^2(\omega))} .
\end{equation}
Finally, by the Agmon and Poincar\'e inequalities and the Cauchy--Schwarz inequality,
\begin{multline*}
	\Big|\int_{0}^{\theta}\langle z_\theta\,\partial_x \lambda_\theta, z_\theta\rangle\d t\Big|\leq\int_0^\theta\norm{z_\theta}_{L^\infty(0,1)}\norm{\lambda_\theta}_{H^1_0(0,1)}\norm{z_\theta}_{L^2(0,1)}\d t\\
	\leq\sup_{t\in[0,\theta]}\|z_\theta(t)\|_{L^2 (0,1)}\, \|z_\theta\|_{L^2 (0,\theta;H_0^1(0,1))}\, \|\lambda_\theta\|_{L^2 (0,\theta;H_0^1(0,1))} ,
\end{multline*}
so that, by \eqref{eq_J_theta_sec_der}, \eqref{lambdaPbound}, \eqref{zP} and \eqref{rthetadef},
\begin{equation*}
	J_\theta^{\prime\prime}(u)[v_\theta ,v_\theta ]\geq \big(1-K_z^2\,C_3^0\sqrt\theta(1+C_2)\,r_\theta\big)\int_{0}^{\theta}\|v_\theta\|_{L^2(\omega)}^2\d t\geq\frac{1}{2}\int_{0}^{\theta}\|v_\theta\|_{L^2(\omega)}^2\d t ,
\end{equation*}
which is \eqref{convexP}. The strict convexity of $J_\theta$ on the convex set $B^{per}_{r_\theta}$ follows. \hfill$\square$

\section{Proof of Proposition~\ref{local_turnpike}}\label{app:C}
Throughout the proof, we impose on $\varepsilon_M$ finitely many smallness conditions, introduced where they are needed; the final choice is recorded in \eqref{choice_epsM} at the end of the proof. First, since $\|y_0\|_{L^2(0,1)}\leq\frac1\pi\|y_0\|_{H^1_0(0,1)}$ by the Poincar\'e inequality, the condition \eqref{smallness_target_datum_varepsilon} with $\varepsilon_M\leq\varepsilon$ (the threshold of Proposition~\ref{local_uniqueness}) implies $\max\{\|y_0\|_{L^2(0,1)},\normC{y_d}\}\leq\varepsilon$; hence both $(OCP)_T$ and $(Per)_\theta$ have unique optimal triples, denoted by  $ \left(y_{T,*},u_{T,*},\lambda_{T,*}\right) $ and $ \left(y_{\theta,*},u_{\theta,*},\lambda_{\theta,*}\right) $ respectively, and the latter is extended by $\theta$-periodicity over the whole real line.
	
	We shall construct, by a fixed point argument, a solution of the optimality system \eqref{maxim}--\eqref{contrT} which is exponentially close to the periodic optimal triple, and then identify it with the (unique) optimal triple of $(OCP)_T$ thanks to Lemma~\ref{lem:convexity}. To this aim, we study the perturbation
	$$ \delta y:= y_{T}-y_{\theta,*}, \quad \delta \lambda:= \lambda_{T}-\lambda_{\theta,*}, \quad \delta u:= u_{T}-u_{\theta,*},$$
	where $(y_T,u_T,\lambda_T)$ denotes a solution of \eqref{maxim}--\eqref{contrT}: subtracting \eqref{maximp}--\eqref{contrP} from \eqref{maxim}--\eqref{contrT}, the triple $(\delta y,\delta\lambda,\delta u)$ solves
	\begin{equation}\label{pertopti0}
		\left\{
		\begin{aligned}
			&\partial_t \delta y  - \partial_{xx} \delta y + \partial_x \left(y_{\theta,*} \delta y\right) + \delta y\,\partial_x \delta y = \chi_\omega \delta \lambda  , \\
			&-\partial_t \delta \lambda  - \partial_{xx} \delta \lambda - y_{\theta,*}\, \partial_x \delta\lambda -\partial_x\lambda_{\theta,*}\, \delta y - \delta y\, \partial_x \delta\lambda = - \delta y , \\
			&\delta y (t,0)=\delta y (t,1)=0,\quad \delta \lambda (t,0)=\delta \lambda (t,1)=0, \\
			&\delta y (0,x) = y_0 (x) - y_{\theta,*}(0,x),\quad \delta \lambda (T,x)=-\lambda_{\theta,*}(T-\lfloor T/\theta\rfloor\theta),
		\end{aligned}
		\right.
	\end{equation}
	for $(t,x)\in (0,T)\times (0,1)$,
	together with $ \delta u = \chi_\omega\delta\lambda $ a.e. (we used the $\theta$-periodicity of $\lambda_{\theta,*}$ in the terminal condition).
	
	\smallskip
	\emph{The underlying linear-quadratic problem.} We set $H:=L^2(0,1)$, $U:=L^2(\omega)$, let $B\in\mathcal L(U,H)$ be the extension-by-zero operator, $B^*h=h|_\omega$, $Q:=I_U$, and, for $t\in[0,T]$,
	\begin{equation*} 
		A(t)z:=\partial_{xx} z-\partial_x (y_{\theta,*}(t,\cdot)z),\quad
		A^*(t)p=\partial_{xx} p+ y_{\theta,*}(t,\cdot)\partial_x p \qquad \forall z,p \in H^2(0,1) \cap H_0^1(0,1),
	\end{equation*}
	as well as
	\begin{equation*}
		G(t)z:=z-\partial_x\lambda_{\theta,*}(t)\,z,\quad \forall z\in H.
	\end{equation*}
	Lemma~\ref{lem:periodic_regularity} gives
	$\|\partial_x\lambda_{\theta,*}\|_{L^\infty((0,\theta)\times(0,1))} \leq C_\alpha \normC{y_d}$. Requiring $C_\alpha\varepsilon_M\leq\frac{1}{2}$, we obtain
	\begin{equation}\label{Gpositive}
		\langle G(t)z,z\rangle_H\geq \frac{1}{2}\|z\|_H^2,\qquad\forall z\in H,\ \text{a.e. }t .
	\end{equation}
	Thus $G(t)$ is a bounded, self-adjoint, uniformly positive operator on $H$; we set $C(t):=G(t)^{1/2}$, so that $G=C^*C$. We now make explicit the uniformity needed below. The operators $A(t)$ have the common dense domain $H^2(0,1)\cap H_0^1(0,1)$. By \eqref{periodic_yH1},
their drift coefficients are uniformly small in $C([0,\theta];H_0^1)$ when $\normC{y_d}$ is small. Multiplication of $z_t=A(t)z$ by $z$, followed by \eqref{ineq:corecive}, therefore gives
$\|U_A(t,s)\|_{\mathcal L(H)}\leq C e^{-\gamma(t-s)}$
with $C,\gamma>0$ uniform on a sufficiently small target ball. Hence $(A(\cdot),B)$ is uniformly exponentially stabilizable (the zero feedback already suffices). Moreover, \eqref{Gpositive} gives $\frac12I\leq C(t)^*C(t)\leq\frac32I$, so $(A(\cdot),C(\cdot))$ is uniformly exponentially detectable. These properties are stable under the above small perturbations; see \cite[Theorem 2.1]{WX}. The periodic Riccati results in \cite[Proposition 3.1 and Theorem 3.8]{D}, \cite[Proposition 3.4]{PI1} and \cite[Subsection 3.3]{TZZ}, applied with these uniform stabilizability and detectability bounds, consequently provide constants $M_1,M_2,M_3,\mu>0$ depending only on $\theta$ and $\omega$. The same constants apply to every time shift of the target, because such a shift only translates the periodic coefficients and leaves all the preceding bounds unchanged. In particular, the differential Riccati equation $\dot{P}+A^*P+PA-PBQ^{-1}B^*P+C^*C=0$ on $(0,T)$, with $P(T)=0$, admits a unique mild solution\footnote{We say that $P(\cdot)\in C([0,T];\Sigma^{+}(H))$ is a mild solution of the differential Riccati equation if, for each $t \in [0,T]$ and $h \in H$, $P(t)h=U_A^*(T,t)P(T)U_A(T,t)h-\int_t^TU_A^* (s,t)\left( P(s)BQ^{-1}B^*P(s)-C^*(s)C(s)\right)U_A(s,t)h \d s$.}  $P_{T}(\cdot) \in C\big([0,T];\Sigma^{+}(H)\big)$, and the corresponding $\theta$-periodic Riccati equation admits a unique $\theta$-periodic mild solution $P_\theta(\cdot) \in C\big([0,\theta];\Sigma^{+}(H)\big)$; moreover,
	\begin{equation}\label{prop1}
		\|P_\theta(t)-P_T(t)\|_{\mathcal{L}\left(H,H\right)} \leq M_1 e^{-2\mu (T-t)}, \quad 0 \leq t \leq T,
	\end{equation}
	\begin{equation}\label{prop2}
		\|U_\theta(t,s)\|_{\mathcal{L}\left(H,H\right)} \leq M_2 e^{-\mu (t-s)} ,\quad  0\leq s \leq t,
	\end{equation}
	where   $U_\theta(\cdot,\cdot)$ is the evolution operator generated by $F_\theta(\cdot)=A(\cdot)-B Q^{-1}B^*P_\theta(\cdot)$. In particular, the uniform constant $M_3$ may be chosen so that
	$$
	\sup_{t\in[0,T]}\|P_T(t)\|_{\mathcal L(H,H)} \leq \sup_{t\in[0,\theta]}\|P_\theta(t)\|_{\mathcal L(H,H)}+M_1
	\leq M_3<\infty.
	$$
	
	\smallskip
	\emph{The fixed point map.} For $ 0<M\leq 1 $, to be fixed below, set
	\begin{equation*}
		X_M:=\big\{ (\tilde{y},\tilde{\lambda})\ \mid\ \|\tilde{y}(t)\|_{H_0^1(0,1)}+\|\tilde{\lambda}(t)\|_{H_0^1(0,1)}\leq M\big(e^{-\mu t }+e^{-\mu(T-t)}\big),\ \text{a.e. } t \in[0,T]\big\}
	\end{equation*}
	and, for $ (\tilde{y},\tilde{\lambda}) \in X_M $, define $R_1(\tilde{y}):= - \tilde{y}\,\partial_x \tilde{y}$ and $R_2(\tilde{y},\tilde{\lambda}):= \tilde{y}\, \partial_x \tilde\lambda$.
	Since, by the Agmon, Young and Poincar\'e inequalities, $\|y_1\partial_x y_2\|_{L^2(0,1)}\leq   \|y_1\|_{H_0^1(0,1)} \| y_2\|_{H_0^1(0,1)}$ for all $y_1,y_2\in H^1_0(0,1)$, we have, for each $ (\tilde{y},\tilde{\lambda}) \in X_M $ and a.e.\ $ t \in[0,T] $,
	\begin{equation}\label{steady_remainesti}
			\|R_1(\tilde{y})(t)\|_{L^2(0,1)}+\|R_2(\tilde{y},\tilde{\lambda})(t)\|_{L^2(0,1)} \leq 4 M^2\big(e^{-2\mu t }+e^{-2\mu(T-t)}\big).
	\end{equation}
	For each $ (\tilde{y},\tilde{\lambda}) \in X_M $, consider the linear system obtained by freezing the quadratic remainders in \eqref{pertopti0}:
	\begin{equation}\label{pertopti}
		\left\{
		\begin{aligned}
			&\partial_t \delta y  - \partial_{xx} \delta y + \partial_x \left(y_{\theta,*}\, \delta y\right) = \chi_\omega \delta{\lambda} + R_1(\tilde{y})  ,\\
			&-\partial_t \delta \lambda  - \partial_{xx} \delta \lambda - y_{\theta,*}\, \partial_x \delta\lambda  = - \delta{y} + \partial_x\lambda_{\theta,*}\, \delta y  + R_2(\tilde{y},\tilde{\lambda}) , \\
			&\delta y (t,0)=\delta y (t,1)=0,\quad \delta \lambda (t,0)=\delta \lambda (t,1)=0, \\
			&\delta y (0,x) = y_0 (x) - y_{\theta,*}(0,x),\quad \delta \lambda (T,x)=-\lambda_{\theta,*}(T-\lfloor T/\theta\rfloor\theta) ,
		\end{aligned}
		\right.
	\end{equation}
	for all $(t,x)\in (0,T)\times (0,1)$.
	The system \eqref{pertopti} is the optimality system of the linear quadratic optimal control problem
	\begin{equation}\label{standardLQ} 
		\inf_{v\in L^2(0,T;L^2(\omega))}\;\frac{1}{2}\int_0^T 
		\big(\langle G(t) \delta y(t), \delta y(t)\rangle +  \|v(t)\|_{L^2(\omega)}^2 - 2\langle R_2(\tilde{y},\tilde{\lambda})(t),\delta y(t)\rangle\big)\d t + \langle \lambda_{\theta,*}(T),\delta y(T)\rangle,
	\end{equation}
	where $\delta y$ solves the state equation in \eqref{pertopti} with $\chi_\omega\delta\lambda$ replaced by $\chi_\omega v$; since $G\geq\frac{1}{2}I$ by \eqref{Gpositive}, this problem is strictly convex and coercive, hence it has a unique minimizer, and the system \eqref{pertopti} has a unique solution $(\delta y,\delta\lambda)$, given by the optimal state and adjoint state of \eqref{standardLQ}. This defines a (single-valued) map
	\begin{equation*}
		\mathcal{K}:X_M\to L^2(0,T;L^2(0,1))\times L^2(0,T;L^2(0,1)),\qquad \mathcal{K}(\tilde{y},\tilde{\lambda}):=(\delta y,\delta \lambda) .
	\end{equation*}
	
	\begin{claim}\label{claim}
		Under a smallness condition on $\|y_0\|_{H^1_0(0,1)}$ and $\normC{y_d}$, the map $ \mathcal{K} $ has a fixed point in $ X_M $ for some appropriate $ M>0$.
	\end{claim}
	We prove Claim~\ref{claim} in Steps 1--3 below; in Steps 4--5 we then show that the fixed point provides a solution of \eqref{maxim}--\eqref{contrT} which coincides with the unique optimal triple of $(OCP)_T$, thereby completing the proof: indeed, for the fixed point, the Poincar\'e inequality and $\delta u=\chi_\omega\delta\lambda$ give
	\begin{equation}\label{conclusionturnpike}
		\|\delta{y}(t)\|_{L^2(0,1)}+\|\delta{\lambda}(t)\|_{L^2(0,1)}+\|\delta{u}(t)\|_{L^2(\omega)}\leq \frac{2}{\pi}\,M\big(e^{-\mu t }+e^{-\mu(T-t)}\big)\leq M\big(e^{-\mu t }+e^{-\mu(T-t)}\big)
	\end{equation}
	for a.e.\ $t\in[0,T]$, which is \eqref{turnpike}.
	
	\medskip
	\textbf{Step 1: $\mathcal K$ is a continuous and compact self-mapping candidate.} The set $X_M$ is a bounded, convex and closed subset of $L^2(0,T;L^2(0,1))^2$: closedness holds because, if a sequence of $X_M$ converges in $L^2(0,T;L^2(0,1))^2$, its uniform $L^\infty(0,T;H^1_0(0,1))$ envelope provides a weak-$\star$ limit in $L^\infty(0,T;H^1_0(0,1))^2$, which is identified with the strong limit and still obeys the pointwise envelope defining $X_M$. Moreover, the right-hand sides in \eqref{pertopti} are bounded in $L^2(0,T;L^2(0,1))$ uniformly with respect to $(\tilde y,\tilde\lambda)\in X_M$ (by \eqref{steady_remainesti}), so that, by parabolic energy estimates applied to each equation separately (the required uniform bound on $(\delta y,\delta\lambda)$ in $C([0,T];L^2(0,1))^2$ is provided by the estimates \eqref{yesti}--\eqref{adesti} of Step~2, which do not depend on the present step), $\mathcal K(X_M)$ is bounded in $W(0,T)^2$, hence relatively compact in $ L^2(0,T;L^2(0,1))^2 $ by the Aubin--Lions lemma. Finally, $\mathcal K$ is continuous on $X_M$ for the $L^2(0,T;L^2(0,1))^2$ topology. Indeed, if $(\tilde y_n,\tilde\lambda_n)\to(\tilde y,\tilde\lambda)$ in this topology, the uniform $L^\infty(0,T;H_0^1(0,1))$ bound and the one-dimensional Agmon inequality give strong convergence in $L^2(0,T;L^\infty(0,1))$. After extraction, the spatial derivatives converge weakly in $L^2(0,T;L^2(0,1))$; consequently, $R_1(\tilde y_n)\rightharpoonup R_1(\tilde y)$ and $R_2(\tilde y_n,\tilde\lambda_n)\rightharpoonup R_2(\tilde y,\tilde\lambda)$ in $L^2(0,T;L^2(0,1))$.
The corresponding solutions of \eqref{pertopti} converge weakly in $W(0,T)^2$ to the solution associated with the limiting sources and strongly in $L^2(0,T;L^2(0,1))^2$ thanks to the compact embedding of $W(0,T)$ into $L^2(0,T;L^2(0,1))$. Uniqueness of the limiting linear system removes the need to extract a subsequence and proves the continuity of $\mathcal K$.
	
	\medskip
	\textbf{Step 2: invariance $ \mathcal{K}(X_M)\subset X_M $ for an appropriate $ M >0$.} Let $ (\tilde{y},\tilde{\lambda}) \in X_M $ and $(\delta y,\delta\lambda)=\mathcal K(\tilde y,\tilde\lambda)$. Set $ h  := -\delta\lambda - P_T\, \delta y$ (in the weak sense). A direct computation using the Riccati equation satisfied by $P_T$ gives the following identity in the pivot space $H$:
	\begin{equation*}
		-\dot h=\big(A^*-P_TBQ^{-1}B^*\big)h+P_TR_1(\tilde y)-R_2(\tilde y,\tilde\lambda),
		\qquad h(T)=\lambda_{\theta,*}(T-\lfloor T/\theta\rfloor\theta).
	\end{equation*}
	No boundary condition is asserted directly for $h$, since $P_T\delta y$ is an $H$-valued Riccati term. Writing $P_TBQ^{-1}B^*=P_\theta BQ^{-1}B^*-(P_\theta-P_T)BQ^{-1}B^*$ and using the evolution operator $U_\theta$ of \eqref{prop2}, one obtains the mild representation
	\begin{multline*}
			h(t)=U^*_{\theta}(T,t)\,\lambda_{\theta,*}(T-\lfloor T/\theta\rfloor\theta) \\+\int_{t}^{T}U^*_\theta(\tau,t)\big((P_\theta(\tau)-P_T(\tau))BQ^{-1}B^*h(\tau) + P_T(\tau) R_1(\tilde{y})(\tau)  - R_2(\tilde{y},\tilde{\lambda})(\tau)\big)\d\tau .
	\end{multline*}
	Set $s:=T-t$ and $\Lambda_0:=\|\lambda_{\theta,*}(T-\lfloor T/\theta\rfloor\theta)\|_{L^2(0,1)}$. By \eqref{prop1}, \eqref{prop2} and \eqref{steady_remainesti}, the function $\phi(s):=e^{\mu s}\|h(T-s)\|_{L^2(0,1)}$ satisfies
	\begin{equation*}
		\phi(s)\leq M_2\Lambda_0+M_1M_2\int_0^se^{-2\mu\tau}\phi(\tau) \d\tau+CM_2M^2\Big(\frac1\mu+\frac{1}{3\mu}e^{-2\mu(T-s)}e^{\mu s}\Big)
	\end{equation*}
	(we used $\int_0^se^{\mu\tau}e^{-2\mu\tau}\d\tau\leq\frac1\mu$ and $\int_0^se^{\mu\tau}e^{-2\mu(T-\tau)}\d\tau\leq\frac{1}{3\mu}e^{-2\mu(T-s)}e^{\mu s}$), so that Gr\"onwall's inequality (the kernel $M_1M_2e^{-2\mu\tau}$ having integral at most $\frac{M_1M_2}{2\mu}$) yields
	\begin{equation}\label{hesti}
		\|h(t)\|_{L^2(0,1)}\leq C(\Lambda_0+M^2) e^{-\mu (T-t)}+ C M^2 e^{-2\mu t},\quad \forall t \in[0,T].
	\end{equation}
	Next, recalling that $\delta\lambda=-h-P_T\delta y$ and rewriting the state equation in \eqref{pertopti} as $\partial_t \delta y-\partial_{xx} \delta y + \partial_x (y_{\theta,*} \delta y) +  \chi_\omega P_\theta \delta y =  \chi_\omega(P_\theta-P_T) \delta y - \chi_\omega h  + R_1(\tilde{y})$, we get
	\begin{equation*}
		\delta y(t)= U_\theta(t,0)(y_0 - y_{\theta,*}(0)) + \int_{0}^{t} U_\theta(t,\tau)\big( \chi_\omega(P_\theta(\tau)-P_T(\tau)) \delta y(\tau) - \chi_\omega h(\tau)  + R_1(\tilde{y})(\tau)\big)\d\tau ,
	\end{equation*}
	and the same weighted Gr\"onwall argument (now with the weight $e^{\mu t}$), combined with \eqref{prop1}, \eqref{prop2}, \eqref{steady_remainesti} and \eqref{hesti}, gives
	\begin{equation}\label{yesti}
			\|\delta y(t)\|_{L^2(0,1)}\leq C(d_0^0 + \Lambda_0 + M^2) \big(e^{-\mu t} + e^{-\mu(T-t)} \big),\qquad d_0^{0}:=\|y_0 - y_{\theta,*}(0)\|_{L^2(0,1)} ,
	\end{equation}
	whence also, using $\delta\lambda = -h - P_T \delta y$ and $\sup_t\|P_T(t)\|_{\mathcal L(H,H)}\leq M_3$,
	\begin{equation}\label{adesti}
		\|\delta\lambda(t)\|_{L^2(0,1)}\leq C(d_0^0 + \Lambda_0 + M^2)\big(e^{-\mu t} + e^{-\mu(T-t)}\big),\qquad \forall t\in[0,T].
	\end{equation}
	
	We now upgrade \eqref{yesti}--\eqref{adesti} to $H^1_0$ norms. First, by \eqref{y_p_H01} and \eqref{uthetasmall},
	\begin{equation}\label{eq_small_y_theta}
		\sup_{t\in[0,\theta]}\|\partial_x y_{\theta,*}(t)\|_{L^2(0,1)}^2
		\leq \Big(4+\frac1\theta\Big) e^{\frac{8\theta}{\pi}\normC{y_d}^2 }\,\theta\,\normC{y_d}^2 ,
	\end{equation}
	so that, requiring $\varepsilon_M$ to be small enough, we may assume
	$ \tilde{\mu}:=1-8\sup_{t\in[0,\theta]}\|\partial_x y_{\theta,*}(t)\|_{L^2(0,1)}^2\geq\frac{1}{2} $.
	Multiplying the state equation in \eqref{pertopti} by $-\partial_{xx}\delta y$ and using, by the Agmon and Poincar\'e inequalities,
	\begin{equation*}
		\|\partial_x(y_{\theta,*}\delta y)\|^2_{L^2}\leq2\|y_{\theta,*}\|^2_{L^\infty}\|\partial_x\delta y\|^2_{L^2}+2\|\partial_xy_{\theta,*}\|^2_{L^2}\|\delta y\|^2_{L^\infty}\leq \frac{8}{\pi}\|\partial_xy_{\theta,*}\|^2_{L^2}\|\partial_x\delta y\|^2_{L^2} ,
	\end{equation*}
	we obtain, by Young's inequality (and $\frac{16}{\pi}\leq8$),
	$$
		\frac{\d}{\d t}\|\delta y(t)\|_{H_0^1(0,1)}^2 +  \|\partial_{xx} \delta y(t)\|_{L^2(0,1)}^2
		\leq 8\|\partial_x y_{\theta,*}(t)\|_{L^2}^2\|\partial_x\delta y(t)\|_{L^2}^2+4\|\delta\lambda(t)\|_{L^2}^2 + 4\|R_1(\tilde y)(t)\|_{L^2}^2 ,
	$$
	whence, by $\|\partial_{xx}\delta y\|^2_{L^2}\geq\|\partial_x\delta y\|^2_{L^2}$ (Poincar\'e) and the definition of $\tilde\mu$, for all $0\leq s\leq t$,
	\begin{equation}\label{eq_dy_H01_decay}
		\|\delta y(t)\|_{H_0^1(0,1)}^2\leq e^{-\tilde{\mu}(t-s)}\|\delta y(s)\|_{H_0^1(0,1)}^2 + 4\int_s^t \big(\|\delta\lambda(\tau)\|_{L^2(0,1)}^2+ \|R_1(\tilde y)(\tau)\|_{L^2(0,1)}^2\big)\d \tau .
	\end{equation}
	Similarly, multiplying by $\delta y$ gives, for all $0\leq s\leq t$,
	\begin{equation}\label{eq_dy_L2H1}
		\tilde{\mu} \int_s^t\|\partial_{x} \delta y(\tau)\|_{L^2(0,1)}^2 \d \tau \leq \|\delta y(s)\|_{L^2(0,1)}^2 +4 \int_s^t \big(\|\delta\lambda(\tau)\|_{L^2(0,1)}^2 + \|R_1(\tilde y)(\tau)\|_{L^2(0,1)}^2\big)\d \tau .
	\end{equation}
	For $t\geq1$, integrating \eqref{eq_dy_H01_decay} with respect to $s$ over $(t-1,t)$ and estimating $\int_{t-1}^t\|\delta y(s)\|^2_{H_0^1} \d s$ by \eqref{eq_dy_L2H1}, we get
	\begin{equation*}
		\|\delta y(t)\|_{H_0^1(0,1)}^2\leq \frac{1}{\tilde{\mu}}\|\delta y(t-1)\|_{L^2(0,1)}^2+ \Big(4+\frac{4}{\tilde{\mu}}\Big)\int_{t-1}^t \big(\|\delta\lambda\|_{L^2(0,1)}^2+ \|R_1(\tilde y)\|_{L^2(0,1)}^2\big)\d s ,
	\end{equation*}
	which, combined with \eqref{steady_remainesti}, \eqref{yesti} and \eqref{adesti}, gives the desired $H^1_0$-estimate for $\delta y(t)$, $t\geq1$; for $t\in[0,1]$, taking $s=0$ in \eqref{eq_dy_H01_decay} gives the same bound with $\|\delta y(0)\|_{H_0^1}=\|y_0-y_{\theta,*}(0)\|_{H_0^1}$ in place of $d_0^0$. The estimate for $\delta\lambda$ is obtained in the same way, working in reversed time: the terminal datum $-\lambda_{\theta,*}(T-\lfloor T/\theta\rfloor\theta)$ belongs to $H^1_0(0,1)$ by Lemma~\ref{lem:periodic_regularity}, and the source terms $-\delta y+\partial_x\lambda_{\theta,*} \delta y+R_2(\tilde y,\tilde\lambda)$ are bounded in $L^2(0,1)$ by $C(d_0^0+\Lambda_0+M^2)(e^{-\mu t}+e^{-\mu(T-t)})$ thanks to \eqref{yesti}, \eqref{steady_remainesti} and $\|\partial_x\lambda_{\theta,*}\|_{L^\infty}\leq\frac{1}{2}$. Altogether, setting
	\begin{equation*}
		d_0:=\|y_0 - y_{\theta,*}(0)\|_{H_0^1(0,1)},\qquad \Lambda:=\|\lambda_{\theta,*}(T-\lfloor T/\theta\rfloor\theta)\|_{H_0^1(0,1)} ,
	\end{equation*}
	there exists $C_\star\geq1$, independent of $T$ and of $M$, such that
	\begin{equation}\label{eq_dydlambdaH01}
			\|\delta y(t)\|_{H_0^1(0,1)}+\|\delta\lambda(t)\|_{H_0^1(0,1)}\leq C_\star (d_0 + \Lambda + M^2) \big(e^{-\mu t} + e^{-\mu(T-t)} \big),\quad\text{a.e. } t \in[0,T].
	\end{equation}
	We now fix $M_0:=\min\big(1, \frac{1}{2C_\star}\big)$ and let $M\in(0,M_0]$, so that $C_\star M^2\leq\frac M2$. By \eqref{eq_small_y_theta} and Lemma~\ref{lem:periodic_regularity}, there is $C''=C''(\theta,\alpha)>0$ such that $d_0+\Lambda\leq C'' (\|y_0\|_{H^1_0(0,1)}+\normC{y_d})\leq C''\varepsilon_M$; hence, requiring $2C_\star C''\varepsilon_M\leq M$, the estimate \eqref{eq_dydlambdaH01} gives $\|\delta y(t)\|_{H^1_0}+\|\delta\lambda(t)\|_{H^1_0}\leq M(e^{-\mu t}+e^{-\mu(T-t)})$, i.e., $ \mathcal{K}(X_M)\subset X_M $.
	
	\medskip
	
	\textbf{Step 3: existence of a fixed point.} By Steps 1--2 and the Schauder fixed point theorem (see \cite[Theorem 2.1, p.~6]{Barbu1993}), there exists $ (\delta{y},\delta{\lambda})\in X_M $ such that $ (\delta{y},\delta{\lambda})= \mathcal{K}(\delta{y},\delta{\lambda}) $. In other words, $(\delta y,\delta\lambda)$ solves the nonlinear system \eqref{pertopti0} (indeed, \eqref{pertopti} with $(\tilde y,\tilde\lambda)=(\delta y,\delta\lambda)$ is precisely \eqref{pertopti0}), and $\|\delta{y}(t)\|_{H_0^1(0,1)}+\|\delta{\lambda}(t)\|_{H_0^1(0,1)}\leq M\big(e^{-\mu t }+e^{-\mu(T-t)}\big)$ a.e.
	
	\medskip
	
	\textbf{Step 4: a refined estimate at the fixed point.} Let $(\delta y,\delta\lambda)$ be the fixed point of Step 3 and let
	\begin{equation*}
		A:=\mathop{\mathrm{ess\,sup}}_{t\in[0,T]}\ \frac{\|\delta y(t)\|_{H_0^1(0,1)}+\|\delta\lambda(t)\|_{H_0^1(0,1)}}{e^{-\mu t }+e^{-\mu(T-t)}}\ \leq M
	\end{equation*}
	be the smallest constant for which $(\delta y,\delta\lambda)\in X_A$. Since $(\tilde y,\tilde\lambda)=(\delta y,\delta\lambda)$, the remainder bounds \eqref{steady_remainesti} hold with $M^2$ replaced by $A^2$, and all the estimates of Step 2 can be run again with this improved input; the outcome \eqref{eq_dydlambdaH01} then reads, by minimality of $A$,
	\begin{equation*}
		A \leq C_\star (d_0+\Lambda+A^2) \leq C_\star (d_0+\Lambda)+\frac{A}{2} ,
	\end{equation*}
	where we used $C_\star A\leq C_\star M\leq\frac{1}{2}$. Hence, setting $C^\flat := 2C_\star C''$,
	\begin{equation*}
		A \leq 2C_\star(d_0+\Lambda) \leq 2C_\star C''\big(\|y_0\|_{H_0^1(0,1)}+\normC{y_d}\big) \leq C^\flat\varepsilon_M .
	\end{equation*}
	In particular, the control $u^\sharp:=u_{\theta,*}+\chi_\omega\delta\lambda$ satisfies, by Lemma~\ref{Energy}(ii) (applied with $R_2=r_\theta$, which is legitimate since $\sqrt\theta\normC{y_d}\leq\sqrt\theta\,\varepsilon\leq r_\theta$ by \eqref{threshold1}) and the Poincar\'e inequality,
	\begin{multline}\label{usharp}
		\|u^\sharp\|_{L^\infty(0,T;L^2(\omega))}\leq \|\lambda_{\theta,*}\|_{C([0,\theta];L^2(0,1))}+\sup_t\|\delta\lambda(t)\|_{L^2(0,1)}\\
		\leq C_\lambda\normC{y_d}+A\leq (C_\lambda+C^\flat)\varepsilon_M=C^\dagger\varepsilon_M ,
	\end{multline}
	with $C_\lambda:=C_3\sqrt\theta\,(C_2\sqrt\theta+1)$ and $C^\dagger=C_\lambda+C^\flat$.
	
	\medskip
	
	\textbf{Step 5: identification with the optimal triple of $(OCP)_T$.} Define
	\begin{equation*}
		\big(y^\sharp,\,u^\sharp,\,\lambda^\sharp\big):=\big(y_{\theta,*}+\delta y, u_{\theta,*}+\chi_\omega\delta\lambda, \lambda_{\theta,*}+\delta\lambda\big) .
	\end{equation*}
	Since $(\delta y,\delta\lambda)$ solves \eqref{pertopti0}, adding \eqref{maximp}--\eqref{contrP} and expanding the nonlinear terms ($y^\sharp\partial_xy^\sharp=y_{\theta,*}\partial_xy_{\theta,*}+\partial_x(y_{\theta,*}\delta y)+\delta y\partial_x\delta y$ and $y^\sharp\partial_x\lambda^\sharp=y_{\theta,*}\partial_x\lambda_{\theta,*}+y_{\theta,*}\partial_x\delta\lambda+\partial_x\lambda_{\theta,*} \delta y+\delta y\,\partial_x\delta\lambda$), one checks that $(y^\sharp,u^\sharp,\lambda^\sharp)$ solves the optimality system \eqref{maxim}--\eqref{contrT}: in particular $y^\sharp(0)=y_0$ and, by the $\theta$-periodicity of $\lambda_{\theta,*}$, $\lambda^\sharp(T)=\lambda_{\theta,*}(T-\lfloor T/\theta\rfloor\theta)+\delta\lambda(T)=0$. Thus $y^\sharp$ is the state associated with the control $u^\sharp$, $\lambda^\sharp$ is the corresponding adjoint state (the unique solution of the linear backward equation in \eqref{maxim} with state $y^\sharp$ and zero terminal datum), and $u^\sharp=\chi_\omega\lambda^\sharp$; by the expression of the first Gateaux derivative (see Subsection~\ref{sec:local_uniqueness}), the latter identity means that
	\begin{equation*}
		{J_T}'(u^\sharp)[v]=\int_0^T\big\langle u^\sharp(t)-\lambda^\sharp(t)|_\omega, v(t)\rangle_{L^2(\omega)}\d t=0\qquad\forall v\in L^2(0,T;L^2(\omega)) .
	\end{equation*}
	Now, requiring $\varepsilon_M\leq r_0$, $C^\dagger\varepsilon_M\leq r_0$ and $K_{1,2}\,\varepsilon_M\leq r_0$ (where $K_{1,2}$ is the constant of Lemma~\ref{uniform_bound_optima}(ii)), we have:
\begin{itemize}[leftmargin=*,parsep=0.5mm,itemsep=0.5mm,topsep=0.5mm]
	\item $\|y_0\|_{L^2(0,1)}\leq\varepsilon_M\leq r_0$ and $\normC{y_d}\leq r_0$, so that Lemma~\ref{lem:convexity}(i) applies and $J_T$ is strictly convex on $B_{r_0}$;
	\item $u^\sharp\in B_{r_0}$, by \eqref{usharp};
	\item the optimal control $u_{T,*}$ of $(OCP)_T$ belongs to $B_{r_0}$ as well, since, by \eqref{uniform_bound_u^T} and the Poincar\'e inequality, 
	$$
	\|u_{T,*}\|_{L^\infty(0,T;L^2(\omega))}\leq K_{1,2}(\|y_0\|_{L^2(0,1)}+\normC{y_d})\leq K_{1,2}\Big(\frac{1}{\pi}\|y_0\|_{H^1_0(0,1)}+\normC{y_d}\Big)\leq K_{1,2}\,\varepsilon_M\leq r_0
	$$ 
(the hypotheses \eqref{ub_small_data} and \eqref{eq_3.19} hold thanks to $\varepsilon_M\leq\varepsilon$ and \eqref{threshold1}).
\end{itemize}
	For every $w\in B_{r_0}$, the segment $[u^\sharp,w]$ is contained in the convex set $B_{r_0}$, along which $J_T$ is convex; hence $J_T(w)\geq J_T(u^\sharp)+{J_T}'(u^\sharp)(w-u^\sharp)=J_T(u^\sharp)$, i.e., $u^\sharp$ minimizes $J_T$ over $B_{r_0}$. Since $u_{T,*}\in B_{r_0}$ minimizes $J_T$ globally, it also minimizes $J_T$ over $B_{r_0}$; by strict convexity of $J_T$ on $B_{r_0}$, the minimizer over $B_{r_0}$ is unique, and therefore
$u^\sharp=u_{T,*}$, $y^\sharp=y_{T,*}$ and $\lambda^\sharp=\lambda_{T,*}$.
	Consequently $(\delta y,\delta\lambda,\delta u)=(y_{T,*}-y_{\theta,*},\lambda_{T,*}-\lambda_{\theta,*},u_{T,*}-u_{\theta,*})$, and \eqref{conclusionturnpike} concludes the proof of \eqref{turnpike}.
	
	It remains to choose $\varepsilon_M$. Collecting the conditions imposed along the proof, it suffices to take
	\begin{equation}\label{choice_epsM}
		\varepsilon_M:=\min\Big(\varepsilon, \varepsilon_\star, r_\alpha, \frac{1}{2C_\alpha}, \varepsilon_\theta, \frac{M}{2C_\star C''}, r_0, \frac{r_0}{C^\dagger}, \frac{r_0}{K_{1,2}}\Big),
	\end{equation}
	where $\varepsilon$ is the threshold \eqref{threshold1} of Proposition~\ref{local_uniqueness}, $\varepsilon_\star$ is the stabilizability/detectability threshold of \cite[Theorem 2.1]{WX}, $r_\alpha$ and $C_\alpha$ are given by Lemma~\ref{lem:periodic_regularity}, $\varepsilon_\theta>0$ is any number such that $8\big(4+\frac1\theta\big)e^{8\theta/\pi}\theta\,\varepsilon_\theta^2\leq\frac{1}{2}$ (which guarantees $\tilde\mu\geq\frac{1}{2}$ by \eqref{eq_small_y_theta}), and $C_\star, C'', r_0, C^\dagger, K_{1,2}$ are as above. \qed

\end{appendices}


\end{document}